\documentclass[reqno, 10pt]{amsart}
\usepackage{}
\usepackage{color}
\usepackage{amsmath}
\usepackage{amsfonts}
\usepackage{amssymb}
\usepackage{amsthm}
\usepackage{newlfont}
\usepackage{graphicx}

\newtheorem{thm}{Theorem}[section]

\newtheorem{lem}[thm]{Lemma}

\theoremstyle{definition}

\theoremstyle{remark}
\newtheorem{rem}[thm]{Remark}

\numberwithin{equation}{section}
\numberwithin{thm}{section}

\newcommand{\ed}{\end {document}}

\newcounter{smalllist}

\title[global solution for MHD]{Global well-posedness for the two dimensional compressible
MHD equations with large data}

\author[D. Bian]{Dongfen Bian}
\address{The Graduate School of China Academy of Engineering Physics,
P. O. Box 2101,\ Beijing,\ China,\  100088,}%
\email{bian\_dongfen@mail.com}
\author[B. Guo]{Boling Guo}
\address{Institute of Applied Physics and Computational Mathematics,  P. O. Box 8009,\ Beijing,\ China,\
100088,}
\email{gbl@iapcm.ac.cn}

\begin{document}
\maketitle
\begin{abstract}
In this paper we are concerned with the global well-posedness for the compressible MHD equations with large data. We show that if the shear viscosity $\mu$ is a positive constant and the bulk viscosity $\lambda$ is the power function of the density, that is, $\lambda(\rho)=\rho^{\beta}$ with $\beta>3$, then the two dimensional compressible
MHD system with the periodic boundary conditions on the torus $\mathbb{T}^2$ have a unique global classical solution $(\rho, u,H)$. In this work we extended the results about compressible Navier-Stokes equations in \cite{Karzhikhov} to compressible MHD equations by applying several new techniques to overcome the coupling between velocity and magnetic field.
\end{abstract}

\textbf{Key words~~~}compressible MHD equations; isentropic fluids; global well-posedness; density-dependent viscosity

\section{Introduction}

In this paper, we consider the following compressible Magnetohydrodynamics (MHD) equations on $\mathbb{T}^2$,
\begin{equation}\label{eq_EP_1}
  \begin{cases}
   \partial_t \rho+\mbox{div}(\rho u)=0,\\[0.2cm]
\partial_t (\rho u)+ \mbox{div}(\rho u\otimes u)+\nabla p=\nabla \times H\times H+\mu\Delta u
   +\nabla((\mu+\lambda(\rho))\mbox{div}u) ,\\[0.2cm]
  \partial_t H- \nabla \times(u\times H)-\nu \Delta H=0,\\[0.2cm]
   \mbox{div}H=0,\\[0.2cm]
   (\rho(x,t),H(x,t),u(x,t))_{|t=0}=(\rho_{0}(x),H_{0}(x),u_{0}(x)),
\end{cases}
\end{equation}
which describes the motion of electrically conducting
media in the presence of a magnetic field.
 Here $\rho,\  u,\  H$ and $ p$
 denote the density, velocity, magnetic field and
  pressure respectively. $\rho_0(x)$, $ u_0(x)$ and $H_0(x)$ are initial values.
The pressure term $p(\rho)$ is assumed to obey the polytropic $\gamma$-law, i.e.
\begin{align} \label{eq_EP_2}
p(\rho) = a \rho^\gamma,
\end{align}
where $a$ is the entropy constant and normalized to be one without loss of generality and $\gamma > 1$ is called the adiabatic index. Also, we assume that the functions $\mu(\rho)$, and $\lambda(\rho)$ are defined on $[0,+\infty)$ and satisfy the conditions
\begin{equation}\label{isentropic}
\mu(\rho)=\mbox{const}>0,\ \lambda(\rho)=\rho^{\beta}, \ \beta>3.
 \end{equation}
 And $\mathbb{T}^2$ is the 2 dimensional torus $[0,1]\times [0,1]$ and $t\in [0,T]$ for any fixed $T>0$.

  For smooth initial data such that the
density $\rho_{0}$ is
  bounded and bounded away from zero (i. e. $0<\underline{\rho} \leq \rho_{0}(x)\leq
  M$), existence and uniqueness of local classical solutions to the compressible Navier-Stokes equations have
  been known for a long time (see the pioneering work of J. Nash \cite{J.N} or the paper of
  N. Itaya \cite{n.i}). Matsumura and Nishida \cite{m} proved
the global well-posedness for compressible Navier-Stokes equations
for smooth data close to equilibrium. The reader may refer to
\cite{c2, d1, d2, d3, d4, d5, d6, hoff, k} for more recent advances
on the subject. Particularly,  Danchin
has obtained several important well-posedness in critical spaces for compressible
Navier-Stokes equations \cite{d1, d2, d3, d4, d5}. Chen-Miao-Zhang \cite{c2} have
 proved the local well-posedness in $\dot{B}^{1}_{2,1}\times (\dot{B}^{0}_{2,1})^{2}$
for the viscous shallow water equations and for compressible
Navier-Stokes equations with density dependent viscosities in the
 Besov spaces $\dot{B}^{\frac{N}{p}}_{p,1}$ \cite{c1}. Bian and Yuan
have obtained local well-posedness in the critical Besov
spaces \cite{Bian} and super critical Besov
spaces \cite{Bian2} for the compressible MHD equations. Concerning the global existence of weak
solutions to compressible Navier-Stokes equations for the large
initial data, readers refer to \cite{B1, B2, lions, a.m}, and refer
to \cite{B3, c2, w, H1} and references therein for the viscous
shallow water equations. Recently, Vaigant-Kazhikhov \cite{Karzhikhov} obtained global well-posedness of strong and classical solution for the compressible Navier-Stokes system with large data and without vacuum. Jiu-Wang-Zhou \cite{xin2}  generalized this result to the case which may
contain vacuums with the periodic boundary conditions on the torus $\mathbb{T}^2$.

 Due to the physical importance and
mathematical challenges, the study on \eqref{eq_EP_1} has attracted many
physicists and mathematicians \cite{1, 2, 3}. Existence and uniqueness of (weak,
strong or smooth) solutions in one dimension can be found in
\cite{4, 5, 6, Kato, 7} and the references cited
therein. Construction of global classical or strong solutions to the Navier-Stokes eqautions in the high-dimensional case was
open, and more difficult for the system \eqref{eq_EP_1} since the velocity and magnetic couple and there are more nonlinear terms.
 This paper is devoted to construct global classical solution for the 2-dimensional compressible MHD equations \eqref{eq_EP_1} with large data.
we extended the results in \cite{Karzhikhov} by applying several new techniques to overcome the coupling between velocity and magnetic field.

  In
the absence of heat conduction, it was proved by Z. P. Xin that any
non-zero smooth solution with initially compact supported density
would blow up in finite time (see \cite{xin}). This result was generalized to the
cases for the non-barotropic compressible Navier-Stokes system with
heat conduction \cite{cho} and for non-compact but rapidly
decreasing at far field initial densities \cite{LO}. As a reasonable
starting point, we will therefore restrict our work to solutions
such that $\rho$ remains positive.

Our result is expressed in the following.

\begin{thm}\label{thm_main}
If the initial data $(\rho_0, u_0, H_0)$ satisfy that \begin{equation}\label{initial data}
\begin{split}
& 0<(\rho_0, p(\rho_0))\in W^{2,q}(\mathbb{T}^2)\times  W^{2,q}(\mathbb{T}^2), \\& (u_0, H_0)\in H^2(\mathbb{T}^2)\times H^2(\mathbb{T}^2),\  \int_{\mathbb{T}^2}\rho_0\mbox{d}x>0.
\end{split}
\end{equation}
Then there exists a unique classical global solution
to the 2D MHD system \eqref{eq_EP_1}--\eqref{isentropic} and satisfies
\begin{equation}
\begin{split}
&0<\rho\leq C,\ (\rho, p(\rho))\in C([0,T]; W^{2,q}(\mathbb{T}^2)),\ \rho_t\in C([0,T];L^q(\mathbb{T}^2)), \\& (u, H)\in C([0,T]; H^2(\mathbb{T}^2))\cap L^2([0,T]; H^3(\mathbb{T}^2)),\\
&(u_t, H_t)\in L^2([0,T]; H^1(\mathbb{T}^2)),\  (\sqrt{\rho}u_t, H_t)\in L^{\infty}([0,T]; L^2(\mathbb{T}^2)).
\end{split}
\end{equation}
\end{thm}

\begin{rem}
We did not make any effort to optimize the assumptions on the initial data or the function space that is
used in the analysis. The whole point is to construct global in time classical solutions.
\end{rem}

\textbf{Notation:} Throughout the paper, $C$ stands for a
"harmless " constant, which is independent of $m$, $t\in [0,T]$. We sometimes use the notation $C_{\alpha}$ for some generic constant depending only on $\alpha$. we use the notation $A\lesssim
B$ as an equivalent of $A\leq CB$. The notation $f_+$ means
that $f_+=\max\{0,f\}$. The notation $L^p(\mathbb{T}^2)$, $1\leq p\leq\infty$, stands for the usual Lebesgue spaces on $\mathbb{T}^2$ and $\|\cdot\|_p$ denotes its $L^p$ norm,  and without ambiguity, we write $\int f(x)\mbox{d}x$ instead of  $\int_{\mathbb{T}^2}f(x)\mbox{d}x$.

\section{Preliminaries}
First, we state some assertions that are used later.

\begin{lem}\label{Bernstein}
 For every function $u\in W_0^{1,m}(\mathbb{T}^2)$ or $u\in W^{1,m}(\mathbb{T}^2)$ with $\int_{\mathbb{T}^2}u\mbox{d}x=0$, it holds that
\begin{align*} \|u\|_{q}\leq C \|\nabla u\|_{m}^\theta\|u\|_{r}^{1-\theta},
\end{align*}
where $\theta=(\frac{1}{r}-\frac{1}{q})(\frac{1}{r}-\frac{1}{m}+\frac{1}{2})^{-1}$, and if $m<2$, then $q$ is between $r$ and $\frac{2m}{2-m}$, that is, $q\in [r, \frac{2m}{2-m}]$ if $r<\frac{2m}{2-m}$, $q\in [\frac{2m}{2-m}, r]$  if $r\geq\frac{2m}{2-m}$, if $m=2$, then $q\in [r,+\infty)$, if $m>2$, then $q\in [r, +\infty]$. Consequently, for every function $u\in W^{1,m}(\mathbb{T}^2)$, one has
\begin{align*} \|u\|_{q}\leq C (\|u\|_1+\|\nabla u\|_{m}^\theta\|u\|_{r}^{1-\theta}),
\end{align*}
where $C$ is a constant which may depend on $q$.
\end{lem}
The above Lemma is the Gagliardo-Nirenberg inequality which can be found in \cite{Novotny, Karzhikhov}. The following Lemma is the Poincare inequality.

\begin{lem}\cite{Yudovich, Pokhozhaev, Talenti} \label{lem2}
For every function $u\in W_0^{1,m}(\mathbb{T}^2)$ or $u\in W^{1,m}(\mathbb{T}^2)$ with $\int_{\mathbb{T}^2}u\mbox{d}x=0$, if $1\leq m<2$, then
\begin{align*} \|u\|_{\frac{2m}{2-m}}\leq C(2-m)^{-\frac{1}{2}} \|\nabla u\|_{m},
\end{align*}
where the positive constant $C$ is independent of $m$.
\end{lem}
From Lemma \ref{lem2}, we can prove the following Lemma, of which proof can also be found in \cite{Karzhikhov}.

\begin{lem} \label{lem3}
For every function $u\in W^{1,\frac{2m}{m+\eta}}(\mathbb{T}^2)$ with $m\geq 2$ and $0<\eta\leq1$, we have
\begin{align*} \|u\|_{2m}\leq C (\|u\|_1+m^{\frac{1}{2}}\|u\|_{2(1-\varepsilon)}^s\|\nabla u\|_{\frac{2m}{m+\eta}}^{1-s}),
\end{align*}
where $\varepsilon\in [0, \frac{1}{2}]$, $s=\frac{(1-\varepsilon)(1-\eta)}{m-\eta(1-\varepsilon)}$ and the positive constant $C$ is independent of $m$.
\end{lem}

By virtue of the maximum principle, we can prove the following lemma. For brevity, we state it here without proof.

\begin{lem} \label{lem_phase}
If the initial datum $|H_0|\leq C$, and the function $H$ satisfies
\begin{align*}
H_t+u\cdot\nabla H-\nu \Delta H-H\cdot\nabla u+H\mbox{div}u=0,
\end{align*}
then $|H|\leq C$, with $C$ a positive constant.
\end{lem}
 Second, we introduce the following notations:
 \begin{align}
 F=(2\mu+\lambda(\rho))\mbox{div}u-p(\rho)-\frac{1}{2}|H|^2, \ w=\partial_{x_1}u_2-\partial_{x_2}u_1,\\
 B=\frac{1}{\rho}(\mu w_{x_1}+(F+\frac{1}{2}|H|^2)_{x_2}), \
 L=\frac{1}{\rho}(-\mu w_{x_2}+(F+\frac{1}{2}|H|^2)_{x_1}).
 \end{align}

Using them and recalling that $\mu=\mbox{const}$, from system \eqref{eq_EP_1}, we get
\begin{align*}
u_{1t}+u\cdot\nabla u_1-\frac{1}{\rho}H\cdot\nabla H_1=\frac{1}{\rho}(-\mu \omega_{x_2}+F_{x_1})=L-\frac{1}{2\rho}|H|^2_{x_1},\\
u_{2t}+u\cdot\nabla u_2-\frac{1}{\rho}H\cdot\nabla H_2=\frac{1}{\rho}(\mu \omega_{x_1}+F_{x_2})=B-\frac{1}{2\rho}|H|^2_{x_2}.
\end{align*}
Then $\omega$ and $F$ solve the following system:
\begin{align}\label{omega and f}
\begin{cases}
\omega_t+u\cdot\nabla\omega+\omega\mbox{div}u+(\frac{1}{\rho}H\cdot\nabla H_1)_{x_2}-(\frac{1}{\rho}H\cdot\nabla H_1)_{x_1}\\ =(B-\frac{1}{2\rho}|H|^2_{x_2})_{x_1}-(L-\frac{1}{2\rho}|H|^2_{x_1})_{x_2},\\
(F+\frac{1}{2}|H|^2)_t+u\cdot\nabla (F+\frac{1}{2}|H|^2)+(2\mu+\lambda(\rho))[(u_{1x_1})^2+2u_{1x_2}u_{2x_1}\\+(u_{2x_2})^2]-\rho(2\mu+\lambda(\rho))
[(F+\frac{1}{2}|H|^2)(\frac{1}{2\mu+\lambda(\rho)})'
+(\frac{p}{2\mu+\lambda(\rho)})']\mbox{div}u\\
-\frac{2\mu+\lambda(\rho)}{\rho}[(H_{1x_1})^2+2H_{1x_2}H_{2x_1}+(H_{2x_2})^2]
-(2\mu+\lambda(\rho))\\ [(\frac{1}{\rho})_{x_1} H \cdot \nabla H_1+(\frac{1}{\rho})_{x_2}H\cdot \nabla H_2]\\ =(2\mu+\lambda(\rho))[(B-\frac{1}{2\rho}|H|^2_{x_2})_{x_2}+(L-\frac{1}{2\rho}|H|^2_{x_1})_{x_1}].
\end{cases}
\end{align}
Using the form of the functions $B(x,t)$ and $L(x,t)$ and the continuity equation, together with the magnetic equation, the system for $(B, L, H)$ can be derived as \begin{align}\label{BL}
\begin{cases}
\rho B_t+\rho u\cdot\nabla B-\rho B\mbox{div}u+u_{x_2}\cdot\nabla(F+\frac{1}{2}|H|^2)+\mu u_{x_1}\nabla\omega+\mu[\omega\mbox{div}u\\+(\frac{1}{\rho}H\cdot\nabla H_1)_{x_2}-(\frac{1}{\rho}H\cdot\nabla H_1)_{x_1}]_{x_1}+\{(2\mu+\lambda(\rho))[(u_{1x_1})^2+2u_{1x_2}u_{2x_1}\\
+(u_{2x_2})^2]\}_{x_2}-\{\frac{2\mu+\lambda(\rho)}{\rho}[(H_{1x_1})^2+2H_{1x_2}H_{2x_1}+(H_{2x_2})^2]\}_{x_2}\\ -\{\rho(2\mu+\lambda(\rho))
[(F+\frac{1}{2}|H|^2)(\frac{1}{2\mu+\lambda(\rho)})'
+(\frac{p}{2\mu+\lambda(\rho)})']\mbox{div}u\}_{x_2}\\-\{(2\mu+\lambda(\rho)) [(\frac{1}{\rho})_{x_1} H \cdot \nabla H_1+(\frac{1}{\rho})_{x_2}H\cdot \nabla H_2]\}_{}x_2\\ =\mu\{(B-\frac{1}{2\rho}|H|^2_{x_2})_{x_1}-(L-\frac{1}{2\rho}|H|^2_{x_1})_{x_2}\}_{x_1}
+\{(2\mu+\lambda(\rho))\\ [(B-\frac{1}{2\rho}|H|^2_{x_2})_{x_2}+(L-\frac{1}{2\rho}|H|^2_{x_1})_{x_1}]\}_{x_2},\\
\rho L_t+\rho u\cdot\nabla L-\rho L\mbox{div}u+u_{x_1}\cdot\nabla(F+\frac{1}{2}|H|^2)-\mu u_{x_2}\nabla\omega-\mu[\omega\mbox{div}u\\+(\frac{1}{\rho}H\cdot\nabla H_1)_{x_2}-(\frac{1}{\rho}H\cdot\nabla H_1)_{x_1}]_{x_2}+\{(2\mu+\lambda(\rho))[(u_{1x_1})^2+2u_{1x_2}u_{2x_1}\\
+(u_{2x_2})^2]\}_{x_1}-\{\frac{2\mu+\lambda(\rho)}{\rho}[(H_{1x_1})^2+2H_{1x_2}H_{2x_1}+(H_{2x_2})^2]\}_{x_1}\\ -\{\rho(2\mu+\lambda(\rho))
[(F+\frac{1}{2}|H|^2)(\frac{1}{2\mu+\lambda(\rho)})'
+(\frac{p}{2\mu+\lambda(\rho)})']\mbox{div}u\}_{x_1}\\-\{(2\mu+\lambda(\rho)) [(\frac{1}{\rho})_{x_1} H \cdot \nabla H_1+(\frac{1}{\rho})_{x_2}H\cdot \nabla H_2]\}_{x_1}\\ =-\mu\{(B-\frac{1}{2\rho}|H|^2_{x_2})_{x_1}-(L-\frac{1}{2\rho}|H|^2_{x_1})_{x_2}\}_{x_2}
+\{(2\mu+\lambda(\rho))\\ [(B-\frac{1}{2\rho}|H|^2_{x_2})_{x_2}+(L-\frac{1}{2\rho}|H|^2_{x_1})_{x_1}]\}_{x_1},\\
H_t+u\cdot\nabla H-\nu\Delta H-H\cdot\nabla u+H\mbox{div}u=0.
\end{cases}
\end{align}
These equations are equivalent to each other for the smooth solutions to the original system \eqref{eq_EP_1}. In the following, we will use the above system in different steps.

\section{ Lower order estimates}
In this section, we derive some uniform a-priori estimates.
Using the method in \cite{Solonnikov, Tani}, we can prove the existence and uniqueness for classical solutions on a sufficiently small time interval. In what follows, we will study the global problem is connected with obtaining a priori estimates with constants only on the data of the problem and duration $T$ of the time interval and independent of the interval of existence of a local solution. Then we can extend this solution globally. We divide the proof of the low order estimates into several steps.

Step 1. Elementary energy estimates

 \begin{lem}\label{lem3.1}
 There exists a positive constant $C$ depending on $(\rho_0, u_0, H_0)$, such that
 \begin{align}\label{elementary estimate}
 \begin{split}
 &\sup_{t\in[0,T]}(\|\sqrt{\rho}u\|_2^2+\|\rho\|_{\gamma}^{\gamma}+\|\rho\|_1+\|H\|_2^2)
 +\int_0^T(\|\omega\|_2^2+\|\nabla u\|_2^2+\|\nabla H\|_2^2\\&
 +\|(2\mu+\lambda(\rho))^{\frac{1}{2}}\mbox{div}u\|_2^2)\mbox{d}t\leq C.
 \end{split}
 \end{align}
 \end{lem}
 \begin{proof}
 Multiplying the momentum equation and the magnetic equation by $u$ and $H$ respectively, integrating over $\mathbb{T}^2$, and then summing the resulting equations, from the continuity equation it holds that
 \begin{align}\label{elementary estimate1}
 \begin{split}
 &\frac{\mbox{d}}{\mbox{d}t}\int(\rho|u|^2+|H|^2)\mbox{d}x+\int(\mu\omega^2+(2\mu+\lambda(\rho))(\mbox{div}u)^2+\nu\|\nabla H\|_2^2)\mbox{d}x\\
 &+\int u\cdot\nabla p\mbox{d}x=0,
 \end{split}
 \end{align}
 where we have used the fact that
 \begin{align*}
 \Delta u=\nabla\mbox{div}u-\nabla\times\omega,
 \end{align*}
 and
  \begin{align*} u\cdot(\nabla\times H\times H)+H\cdot (\nabla\times(u\times H))=\mbox{div}((u\times H)\times H).
 \end{align*}
Note that
\begin{align}\label{gamma}
\int u\cdot \nabla p\mbox{d}x=\frac{\mbox{d}}{\mbox{d}t}\int\frac{\rho^{\gamma}}{\gamma-1}\mbox{d}x.
\end{align}

Integrating the continuity equation, we have
\begin{align}
\frac{\mbox{d}}{\mbox{d}t}\int\rho\mbox{d}x=0,
\end{align}
 which together with \eqref{elementary estimate1} and \eqref{gamma}, gives that
  \begin{align*}
 \begin{split}
 &\frac{\mbox{d}}{\mbox{d}t}\int(\rho|u|^2+|H|^2+\rho+
 \frac{\rho^{\gamma}}{\gamma-1})\mbox{d}x+\int(\mu\omega^2+(2\mu+\lambda(\rho))(\mbox{div}u)^2+\nu\|\nabla H\|_2^2)\mbox{d}x=0,
 \end{split}
 \end{align*}
 Integrating the above equality in the time variable $t$ over $[0,T]$, and using the fact $\|\nabla u\|_2\leq C(\|\omega\|_2+\|\mbox{div}u\|_2)$, it follows that
  \begin{align*}
 \begin{split}
 &\sup_{t\in[0,T]}(\|\sqrt{\rho}u\|_2^2+\|\rho\|_{\gamma}^{\gamma}+\|\rho\|_1+\|H\|_2^2)
 +\int_0^T(\|\omega\|_2^2+\|\nabla u\|_2^2+\|\nabla H\|_2^2\\&
 +\|(2\mu+\lambda(\rho))^{\frac{1}{2}}\mbox{div}u\|_2^2)\mbox{d}t\leq C.
 \end{split}
 \end{align*}
 \end{proof}
 Step 2. Density estimate

 Applying the operator $\mbox{div}$ to the momentum equation of \eqref{eq_EP_1}, we have
 \begin{align}\label{elliptic system}
 [\mbox{div}(\rho u)]_t+\mbox{div}[\mbox{div}(\rho u\otimes u-H\otimes H)]=\Delta F.
 \end{align}
 Consider the following two elliptic problems:
 \begin{align}\label{elliptic system 1}
 \Delta \xi=\mbox{div}(\rho u),\  \int \xi\mbox{d}x=0,
 \end{align}
 \begin{align}\label{elliptic system 2}
 \Delta \eta=\mbox{div}[\mbox{div}(\rho u\otimes u-H\otimes H)],\  \int\eta \mbox{d}x=0,
 \end{align}
 both with the periodic boundary condition on the torus $\mathbb{T}^2$.

 From the elliptic estimates and H\"{o}lder inequality, it can be derived as that
 \begin{lem}\label{lem3.2}
 (1) $ \|\nabla \xi\|_{2m} \leq Cm\|\rho\|_{\frac{2mk}{k-1}}\|u\|_{2mk}$, for any $k>1$, $m\geq 1$;\\
 (2) $ \|\nabla \xi\|_{2-r} \leq C\|\rho\|^{\frac{1}{2}}_{\frac{2-r}{r}}\|\sqrt{\rho}u\|_{2}$, for any $0<r<1$;\\
 (3) $ \|\eta\|_{2m} \leq Cm\|\rho\|_{\frac{2mk}{k-1}}\|u\|_{4mk}^2+Cm\|H\|_{4m}^2$, for any $k>1$, $m\geq 1$;\\
 where $C$ is a positive constant independent of $m$, $k$ and $r$.
 \end{lem}
 \begin{proof}
 By the elliptic estimates to the equation \eqref{elliptic system 2} and using the H\"{o}lder inequality, we have for any $k>1$, $m\geq1$,
 \begin{align*}
 \|\eta\|_{2m} \leq Cm\|\rho u^2\|_{2m}+Cm\|H^2\|_{2m}\l\hat{}eq Cm\|\rho\|_{\frac{2mk}{k-1}}\|u\|_{4mk}^2+Cm\|H\|_{4m}^2.
 \end{align*}
the statements (2) and (3) can be proved in the same way as (1).
 \end{proof}
 Denote \begin{align}\label{psi} \phi(t)=\int(\mu\omega^2+(2\mu+\lambda(\rho))(\mbox{div}u)^2+\nu\|\nabla H\|_2^2)\mbox{d}x,
  \end{align}
  and recall Lemmas \ref{Bernstein}-\ref{lem3} and Lemma \ref{lem3.2}, the following lemma holds.
 \begin{lem}\label{lem3.3}
 (1) $ \| \xi\|_{2m} \leq Cm^{\frac{1}{2}}\|\nabla\xi\|_{\frac{2m}{m+1}}\leq Cm^{\frac{1}{2}}\|\rho\|_{m}^{\frac{1}{2}}$, for any $m\geq2$;\\
 (2) $ \|u\|_{2m} \leq C[m^{\frac{1}{2}}\phi(t)^{\frac{1}{2}}+1]$, for any $m\geq2$;\\
 (3) $ \|\nabla \xi\|_{2m} \leq C [m^{\frac{3}{2}}k^{\frac{1}{2}}\|\rho\|_{\frac{2mk}{k-1}}\phi(t)^{\frac{1}{2}}+m\|\rho\|_{\frac{2mk}{k-1}}]$, for any $k>1$, $m\geq 1$;\\
 (4) $ \|\eta\|_{2m} \leq C[m^2k\|\rho\|_{\frac{2mk}{k-1}}\phi(t)+m\|\rho\|_{\frac{2mk}{k-1}}+m^2\phi(t)+m]$, for any $k>1$, $m\geq 1$;\\
 where $C$ is a positive constant independent of $m$ and $k$.
 \end{lem}
 \begin{proof}

 (1) From Lemma \ref{lem2}, Lemma \ref{elementary estimate} and Lemma \ref{lem3.2}(2), clearly
  \begin{align*}
  \| \xi\|_{2m} \leq Cm^{\frac{1}{2}}\|\nabla\xi\|_{\frac{2m}{m+1}}\leq Cm^{\frac{1}{2}}\|\rho\|_{m}^{\frac{1}{2}}\|\sqrt{\rho}u\|_2\leq C m^{\frac{1}{2}}\|\rho\|_{m}^{\frac{1}{2}}.
 \end{align*}

(2) By Lemma \ref{lem3}, we get
\begin{align}\label{density3.14}
  \| u\|_{2m} \leq C(\|u\|_1+m^{\frac{1}{2}}\|\nabla u\|_{\frac{2m}{m+1}}).
 \end{align}
 Denote $\bar{u}=\|u\|_1=\int u\mbox{d}x$, then
 \begin{align}\label{density3.15}
 |\int\rho (u-\bar{u})\mbox{d}x|\leq \|\rho\|_{\gamma}\|u-\bar{u}\|_{\frac{\gamma}{\gamma-1}}\leq C\|\nabla u\|_2.
 \end{align}
 On the other hand, from the conservative form of the compressible MHD equations \eqref{eq_EP_1} and the periodic boundary conditions, we obtain
 \begin{align*}
 \frac{\mbox{d}}{\mbox{d}t}\int \rho(x,t)\mbox{d}x= \frac{\mbox{d}}{\mbox{d}t}\int \rho u(x,t)\mbox{d}x=0,
 \end{align*}
 that is,
  \begin{align*}
\int \rho(x,t)\mbox{d}x=\int \rho_0(x)\mbox{d}x, \ \int \rho u(x,t)\mbox{d}x=\int \rho_0u_0(x)\mbox{d}x,
 \end{align*}
 for any $t\in [0,T]$.
 Thus,
 \begin{align*}
 |\int\rho(u-\bar{u})\mbox{d}x|= |\int\rho_0u_0\mbox{d}x-\bar{u}\int \rho_0(x)\mbox{d}x|\geq |\bar{u}|\int \rho_0(x)\mbox{d}x-|\int \rho_0u_0(x)\mbox{d}x|,
 \end{align*}
 which together with \eqref{density3.15} implies that
 \begin{align*}
 |\bar{u}|\leq\frac{\int \rho_0u_0(x)\mbox{d}x}{\int \rho_0(x)\mbox{d}x}+\frac{C\|\nabla u\|_2}{\int \rho_0(x)\mbox{d}x}.
 \end{align*}
Substituting the above inequality into \eqref{density3.14} completes the proof of Lemma \ref{lem3.3}(2).

(3) By Lemma (\ref{lem3.2})(1) and Lemma \ref{lem3.3}(2), we can show that
\begin{align*}
&\|\nabla \xi\|_{2m} \leq Cm \|\rho\|_{\frac{2mk}{k-1}}\|u\|_{2mk}\leq Cm \|\rho\|_{\frac{2mk}{k-1}}(m^{\frac{1}{2}}k^{\frac12}\|\nabla u\|_2+1)\\&
\leq C [m^{\frac{3}{2}}k^{\frac{1}{2}}\|\rho\|_{\frac{2mk}{k-1}}\phi(t)^{\frac{1}{2}}+m\|\rho\|_{\frac{2mk}{k-1}}].
\end{align*}

(4) From Lemma (\ref{lem3.2})(3) and Lemma \ref{lem3.3}(2), we have
\begin{align*}
 \|\eta\|_{2m} &\leq Cm\|\rho\|_{\frac{2mk}{k-1}}\|u\|_{4mk}^2+Cm\|H\|_{4m}^2 \\&\leq Cm \|\rho\|_{\frac{2mk}{k-1}}(mk\|\nabla u\|_2^2+1)+Cm^2\|\nabla H\|_2^2+Cm\\
 &\leq C[m^2k\|\rho\|_{\frac{2mk}{k-1}}\phi(t)+m\|\rho\|_{\frac{2mk}{k-1}}+m^2\phi(t)+m].
\end{align*}
 \end{proof}
Combining \eqref{elliptic system 1}-\eqref{elliptic system 2} with \eqref{elliptic system}, we obtain
\begin{align*}
\Delta(\xi_t+\eta-F+\int F(x,t)\mbox{d}x)=0.
\end{align*}
Thus,
\begin{align*}
\xi_t+\eta-F+\int F(x,t)\mbox{d}x=0.
\end{align*}
Since $F=(2\mu+\lambda(\rho))\mbox{div}u-p-\frac12|H|^2$, the preceding equality implies that
\begin{align*}
\xi_t-(2\mu+\lambda(\rho))\mbox{div}u+p+\frac12|H|^2+\eta+\int F(x,t)\mbox{d}x=0.
\end{align*}
Using the continuity equation, one gets
\begin{align}\label{xi1}
\xi_t+(2\mu+\lambda(\rho))\frac{1}{\rho}(\rho_t+u\cdot\nabla \rho)+p(\rho)+\frac12|H|^2+\eta+\int F(x,t)\mbox{d}x=0.
\end{align}
Define
\begin{align*}
\theta(\rho)=\int_1^{\rho}\frac{2\mu+\lambda(s)}{s}\mbox{d}s=2\mu\ln\rho+\frac{1}{\beta}(\rho^{\beta}-1),
\end{align*}
 then \eqref{xi1} yields the following transport equation
 \begin{align}\label{xi2}
(\xi+\theta(\rho))_t+u\cdot\nabla(\xi+\theta(\rho))+p(\rho)+\frac12|H|^2+\eta-u\cdot\nabla\xi+\int F(x,t)\mbox{d}x=0.
\end{align}
 \begin{lem}\label{lem3.4}
 For any $k\geq 1$, it holds that
 \begin{align*}
 \sup_{t\in[0,T]}\|\rho(\cdot, t)\|_k\leq Ck^{\frac{2}{\beta-1}}.
 \end{align*}
 \end{lem}
 \begin{proof}
 Multiplying the equation \eqref{xi2} by the function $\rho[(\xi+\theta(\rho))_+]^{2m-1}$ with $m\geq 4$ a natural number, and integrating the result equation over $\Omega$, involving the continuity equation, we get
 \begin{align}\label{density eq1}
 \begin{split}
 &\frac{1}{2m}\frac{\mbox{d}}{\mbox{d}t}\int\rho[(\xi+\theta(\rho))_+]^{2m}\mbox{d}x+\int\rho p(\rho)[(\xi+\theta(\rho))_+]^{2m-1}\mbox{d}x\\
 &+\frac12\int\rho|H|^2[(\xi+\theta(\rho))_+]^{2m-1}\mbox{d}x\\
 &=-\int\rho\eta[(\xi+\theta(\rho))_+]^{2m-1}\mbox{d}x+\int\rho u\cdot\nabla \xi[(\xi+\theta(\rho))_+]^{2m-1}\mbox{d}x\\
 &-\int F(x,t)\mbox{d}x\int\rho[(\xi+\theta(\rho))_+]^{2m-1}\mbox{d}x.
 \end{split}
 \end{align}
 Put
 \begin{align*}
 f(t)=\{\int\rho[(\xi+\theta(\rho))_+]^{2m}\}^{\frac{1}{2m}},\ t\in [0,T].
 \end{align*}
 Now we estimate the terms on the right-hand side of \eqref{density eq1}.

 First of all,
 \begin{align}\label{density eq2}
 \begin{split}
 &|-\int\rho \eta[(\xi+\theta(\rho))_+]^{2m-1}\mbox{d}x|\\&\leq \int \rho^{\frac{1}{2m}}|\eta|[\rho(\xi+\theta(\rho))_+^{2m}]^{\frac{2m-1}{2m}}\mbox{d}x\\
 &\leq \|\rho\|_{2m\beta+1}^{\frac{1}{2m}}\|\eta\|_{2m+\frac{1}{\beta}}\|\rho(\xi+\theta(\rho))_+^{2m}\|_1^{\frac{2m-1}{2m}}\\
 & \leq C\|\rho\|_{2m\beta+1}^{\frac{1}{2m}}f(t)^{2m-1}[(m+\frac{1}{2\beta})^2k\|\rho\|_{\frac{2(m+\frac{1}{2\beta})k}{k-1}}\|\nabla u\|_2^2+m\|\rho\|_{\frac{2(m+\frac{1}{2\beta})k}{k-1}}\\&+m^2\|\nabla H\|_2^2+m]\\
 &\leq C\|\rho\|_{2m\beta+1}^{1+\frac{1}{2m}}f(t)^{2m-1}[m^2\phi(t)+m]+\|\rho\|_{2m\beta+1}^{\frac{1}{2m}}f(t)^{2m-1}[m^2\phi(t)+m],
 \end{split}
 \end{align}
 where $\phi$ is defined as \eqref{psi}, and we have chosen $k=\frac{\beta}{\beta-1}$.

Next, for $\frac{1}{2m\beta+1}+\frac{2}{p}=1$ with $p\geq1$, the second term on the right-hand side of \eqref{density eq1} can be estimated as
\begin{align}\label{density eq4}
\begin{split}
&|\int\rho u\cdot\nabla \xi[(\xi+\theta(\rho))_+]^{2m-1}\mbox{d}x|\\
&\leq\int\rho^{\frac{1}{2m}}|u||\nabla \xi|
[\rho(\xi+\theta(\rho))_+^{2m}]^{\frac{2m-1}{2m}}\mbox{d}x\\
&\leq C\|\rho\|_{2m\beta+1}^{\frac{1}{2m}}\|u\|_{2mp}\|\nabla\xi\|_{2mp}\|\rho(\xi+\theta(\rho))_+^{2m}\|_1^{\frac{2m-1}{2m}}\\
&\leq C\|\rho\|_{2m\beta+1}^{\frac{1}{2m}}[(mp)^{\frac12}\|\nabla u\|_2+1][(mq)^{\frac32}k^{\frac12}\|\rho\|_{\frac{2mqk}{k-1}}\phi(t)^{\frac12}\\
&+mq\|\rho\|_{\frac{2mqk}{k-1}}]f(t)^{2m-1}\\
&\leq C\|\rho\|_{2m\beta+1}^{1+\frac{1}{2m}}f(t)^{2m-1}[m^{\frac12}\phi^{\frac12}(t)+1][m^{\frac32}\phi(t)^{\frac12}+m]\\
&\leq C\|\rho\|_{2m\beta+1}^{1+\frac{1}{2m}}f(t)^{2m-1}[m^{2}\phi(t)+m],
\end{split}
\end{align}
where in the third inequality we have chosen $p=q=\frac{2m\beta+1}{m\beta}$ and $k=\frac{\beta}{\beta-2}$.

Now let's estimate the last term on the right-hand side of \eqref{density eq1}. From the elementary estimate \eqref{elementary estimate}, one can show that
\begin{align}\label{density eq5}
\begin{split}
&|-\int F(x,t)\mbox{d}x\int\rho[(\xi+\theta(\rho))_+]^{2m-1}\mbox{d}x|\\
&\leq \int |(2\mu+\lambda(\rho))\mbox{div}u-p-\frac12|H|^2|\mbox{d}x
\int\rho^{\frac{1}{2m}}[\rho(\xi+\theta(\rho))_+^{2m}]^{\frac{2m-1}{2m}}\mbox{d}x\\
&\leq \big[\big(\int(2\mu+\lambda(\rho))(\mbox{div}u)^2\mbox{d}x\big)^{\frac12}\big(\int(2\mu+\lambda(\rho))\mbox{d}x\big)^{\frac12}\big)+\int p(\rho)\mbox{d}x\\
&+\frac12\int |H|^2\mbox{d}x\big]\|\rho\|_1^{\frac{1}{2m}}\|\rho(\xi+\theta(\rho))_+^{2m}\|_1^{\frac{2m-1}{2m}}\\
&\leq C[\phi(t)^{\frac12}+\phi(t)^{\frac12}(\int\rho^{\beta}\mbox{d}x)^{\frac12}+1]f(t)^{2m-1}\\
&\leq C[\phi(t)^{\frac12}+\phi(t)^{\frac12}\|\rho\|_{2m\beta+1}^{\frac{\beta}{2}}+1]f(t)^{2m-1}.
\end{split}
\end{align}
 Plugging \eqref{density eq2}-\eqref{density eq5} into \eqref{density eq1} yields that
\begin{align*}
&\frac{1}{2m}\frac{\mbox{d}}{\mbox{d}t}(f^{2m}(t))+\int\rho p(\rho)[(\xi+\theta(\rho))_+]^{2m-1}\mbox{d}x\\
 &\leq C\|\rho\|_{2m\beta+1}^{1+\frac{1}{2m}}f(t)^{2m-1}[m^{2}\phi(t)+m]+C\|\rho\|_{2m\beta+1}^{\frac{1}{2m}}f(t)^{2m-1}[m^2\phi(t)+m]\\
 &+C\|\rho\|_{2m\beta+1}^{\frac{1}{2m}}[m^{\frac12}\phi^{\frac12}+1]f(t)^{2m-1}
 +C[\phi(t)^{\frac12}+\phi(t)^{\frac12}\|\rho\|_{2m\beta+1}^{\frac{\beta}{2}}+1]f(t)^{2m-1}.
\end{align*}
Then it holds that
\begin{align*}
&\frac{\mbox{d}}{\mbox{d}t}f(t)
\leq C[1+\phi(t)^{\frac12}+\phi(t)^{\frac12}\|\rho\|_{2m\beta+1}^{\frac{\beta}{2}}+(m^2\phi(t)+m)\|\rho\|_{2m\beta+1}^{1+\frac{1}{2m}}
 \\
 &+(m^2\phi(t)+m)\|\rho\|_{2m\beta+1}^{\frac{1}{2m}}+
 (m^{\frac12}\phi(t)^{\frac12}+1)\|\rho\|_{2m\beta+1}^{\frac{1}{2m}}].
\end{align*}
Integrating the above inequality over $[0,t]$ gives that
\begin{align}\label{density eq6}
&f(t)
\leq f(0)+ C[1+\int_0^t\phi(s)^{\frac12}\|\rho\|_{2m\beta+1}^{\frac{\beta}{2}}\mbox{d}s+\int_0^t(m^2\phi(s)+m)\|\rho\|_{2m\beta+1}^{1+\frac{1}{2m}}\mbox{d}s
 \\
 &+\int_0^t(m^2\phi(s)+m)\|\rho\|_{2m\beta+1}^{\frac{1}{2m}}\mbox{d}s+
 \int_0^t(m^{\frac12}\phi(s)^{\frac12}+1)\|\rho\|_{2m\beta+1}^{\frac{1}{2m}}\mbox{d}s].
\end{align}
Now we calculate the quantity
\begin{align*}
f(0)=(\int\rho_0[(\xi_0+\theta(\rho_0))_+]^{2m}\mbox{d}x)^{\frac{1}{2m}}.
\end{align*}

From Lemma \ref{lem3.2}(1), it is not difficulty to prove
\begin{align*}
\|\xi_0\|_{L^{\infty}}\leq C.
\end{align*}
Moreover, from the definition of $\theta(\rho_0)=2\mu\ln\rho_0+\frac{1}{\beta}((\rho_0)^{\beta}-1)$, one has
\begin{align*}
\xi_0+\theta(\rho_0)\rightarrow-\infty, \ \mbox{as}\ \rho_0\rightarrow 0^+.
\end{align*}
So there exists a positive constant $\sigma$, such that if $0<\rho_0\leq\sigma$, then
\begin{align*}
(\xi_0+\theta(\rho_0))_+\equiv0.
\end{align*}
Thus it holds that
\begin{align*}
&f(0)=[(\int_{[a\leq\rho_0\leq\sigma]}+\int_{[\sigma\leq\rho_0\leq M]})\rho_0(\xi_0+\theta(\rho_0))_+^{2m}\mbox{d}x]^{\frac{1}{2m}}\\
&=[\int_{[\sigma\leq\rho_0\leq M]}\rho_0(\xi_0+\theta(\rho_0))_+^{2m}\mbox{d}x]^{\frac{1}{2m}}\leq C_{\sigma,M},
\end{align*}
with $C_{\sigma,M}$ a positive constant independent of $m$, which together with \eqref{density eq6} leads to
\begin{align}\label{density eq7}
\begin{split}
&f(t)
\leq  C[1+\int_0^t\phi(s)^{\frac12}\|\rho\|_{2m\beta+1}^{\frac{\beta}{2}}\mbox{d}s+\int_0^t(m^2\phi(s)+m)\|\rho\|_{2m\beta+1}^{1+\frac{1}{2m}}\mbox{d}s
 \\
 &+\int_0^t(m^2\phi(s)+m)\|\rho\|_{2m\beta+1}^{\frac{1}{2m}}\mbox{d}s+
 \int_0^t(m^{\frac12}\phi(s)^{\frac12}+1)\|\rho\|_{2m\beta+1}^{\frac{1}{2m}}\mbox{d}s].
 \end{split}
\end{align}
Set $\Omega_1(t)=\{x\in \mathbb{T}^2|\rho(x,t)>2\}$ and $\Omega_2(t)=\{x\in\Omega_1(t)|\xi(x,t)+\theta(\rho)(x,t)>0\}$. Then one has
\begin{align*}
&\|\rho\|_{2m\beta+1}^{\beta}=(\int\rho^{2m\beta+1}\mbox{d}x)^{\frac{\beta}{2m\beta+1}}
=(\int_{\Omega_1(t)}\rho^{2m\beta+1}\mbox{d}x+\int_{\mathbb{T}^2/ \Omega_1(t)}\rho^{2m\beta+1}\mbox{d}x)^{\frac{\beta}{2m\beta+1}}\\
&\leq C(\int_{\Omega_1(t)}\rho^{2m\beta+1}\mbox{d}x)^{\frac{\beta}{2m\beta+1}}+C \leq C (\int_{\Omega_1(t)}\rho|\theta(\rho)|^{2m}\mbox{d}x)^{\frac{\beta}{2m\beta+1}} +C\\
&\leq C(\int_{\Omega_2(t)}\rho|\theta(\rho)+\xi-\xi|^{2m}\mbox{d}x
+\int_{\Omega_1(t)/\Omega_2(t)}\rho|\theta(\rho)|^{2m}\mbox{d}x)^{\frac{\beta}{2m\beta+1}} +C\\
&\leq C(\int_{\Omega_2(t)}\rho|\theta(\rho)+\xi|^{2m}\mbox{d}x+\int_{\Omega_2(t)}\rho|\xi|^{2m}\mbox{d}x
+\int_{\Omega_1(t)/\Omega_2(t)}\rho|\xi|^{2m}\mbox{d}x)^{\frac{\beta}{2m\beta+1}} +C\\
&\leq C(f(t)^{2m}+\int_{\mathbb{T}^2}\rho|\xi|^{2m}\mbox{d}x)^{\frac{\beta}{2m\beta+1}}+C\leq C[f(t)+(\int_{\mathbb{T}^2}\rho|\xi|^{2m}\mbox{d}x)^{\frac{\beta}{2m\beta+1}}+1].
\end{align*}
Notice that
\begin{align*}
(\int_{\mathbb{T}^2}\rho|\xi|^{2m}\mbox{d}x)^{\frac{\beta}{2m\beta+1}}&\lesssim \|\rho\|_{2m\beta+1}^{\frac{\beta}{2m\beta+1}}\||\xi|^{2m}\|_{\frac{2m\beta+1}{2m\beta}}^{\frac{\beta}{2m\beta+1}}
\lesssim \|\rho\|_{2m\beta+1}^{\frac{\beta}{2m\beta+1}}\|\xi\|_{2m+\frac{1}{\beta}}^{\frac{2m\beta}{2m\beta+1}}\\
&\lesssim \|\rho\|_{2m\beta+1}^{\frac{\beta}{2m\beta+1}}[(m+\frac{1}{2\beta})^{\frac12}
\|\rho\|_{m+\frac{1}{2\beta}}^{\frac12}]^{\frac{2m\beta}{2m\beta+1}}
\\
&\lesssim m^{\frac12}\|\rho\|_{2m\beta+1}^{\frac{\beta(m+1)}{2m\beta+1}}.
\end{align*}
Then from Young inequality, it follows that
\begin{align*}
&\|\rho\|_{2m\beta+1}^{\beta}=(\int\rho^{2m\beta+1}\mbox{d}x)^{\frac{\beta}{2m\beta+1}}
\leq C[1+f(t)+m^{\frac12}\|\rho\|_{2m\beta+1}^{\frac{\beta(m+1)}{2m\beta+1}}]\\
&\leq \frac12\|\rho\|_{2m\beta+1}^{\beta}+C(1+f(t)+m^{\frac{m\beta+\frac12}{m(2\beta-1)}}).
\end{align*}
Thus one can get
\begin{align*}
&\|\rho\|_{2m\beta+1}^{\beta}
\leq C[f(t)+m^{\frac{\beta}{2\beta-1}}]\\
&\leq C[m^{\frac{\beta}{2\beta-1}}+\int_0^t\phi(s)^{\frac12}\|\rho\|_{2m\beta+1}^{\frac{\beta}{2}}\mbox{d}s+\int_0^t(m^2\phi(s)+m)\|\rho\|_{2m\beta+1}^{1+\frac{1}{2m}}\mbox{d}s
 \\
 &+\int_0^t(m^2\phi(s)+m)\|\rho\|_{2m\beta+1}^{\frac{1}{2m}}\mbox{d}s+
 \int_0^t(m^{\frac12}\phi(s)^{\frac12}+1)\|\rho\|_{2m\beta+1}^{\frac{1}{2m}}\mbox{d}s]\\
 &\leq C[m^{\frac{\beta}{2\beta-1}}+\int_0^t\|\rho\|_{2m\beta+1}^{\beta}\mbox{d}s+\int_0^t(m^2\phi(s)+m)\|\rho\|_{2m\beta+1}^{1+\frac{1}{2m}}\mbox{d}s
 \\
 &+\int_0^t(m^2\phi(s)+m)\|\rho\|_{2m\beta+1}^{\frac{1}{2m}}\mbox{d}s+
 \int_0^t(m^{\frac12}\phi(s)^{\frac12}+1)\|\rho\|_{2m\beta+1}^{\frac{1}{2m}}\mbox{d}s].
\end{align*}
From Gronwall's inequality, it can be derived as
\begin{align*}
\|\rho\|_{2m\beta+1}^{\beta}
&\leq C[m^{\frac{\beta}{2\beta-1}}+\int_0^t(m^2\phi(s)+m)\|\rho\|_{2m\beta+1}^{1+\frac{1}{2m}}\mbox{d}s
 +\int_0^t(m^2\phi(s)+m)\|\rho\|_{2m\beta+1}^{\frac{1}{2m}}\mbox{d}s\\
 &+
 \int_0^t(m^{\frac12}\phi(s)^{\frac12}+1)\|\rho\|_{2m\beta+1}^{\frac{1}{2m}}\mbox{d}s].
\end{align*}
Put
\begin{align*}
y(t)=m^{-\frac{2}{\beta-1}}\|\rho\|_{2m\beta+1}(t),
\end{align*}
then
\begin{align*}
y^{\beta}(t)&\leq C[m^{\frac{\beta(1-3\beta)}{(2\beta-1)(\beta-1)}}
+m^{\frac{1}{m(\beta-1)}}\int_0^t\phi(s)y(s)^{1+\frac{1}{2m}}\mbox{d}s
\\
&+m^{\frac{1}{m(\beta-1)}-1}\int_0^ty(s)^{1+\frac{1}{2m}}\mbox{d}s
+m^{\frac{1-2m}{m(\beta-1)}}\int_0^t\phi(s)y(s)^{\frac{1}{2m}}\mbox{d}s
\\
&+m^{\frac{1-m-m\beta}{m(\beta-1)}}\int_0^ty(s)^{\frac{1}{2m}}\mbox{d}s
+m^{\frac{1-2m}{m(\beta-1)}-\frac32}\int_0^t\phi(s)^{\frac12}y(s)^{\frac{1}{2m}}\mbox{d}s
\\
&+m^{\frac{1-2m\beta}{m(\beta-1)}}\int_0^ty(s)^{\frac{1}{2m}}\mbox{d}s]\\
&\leq C[1+\int_0^t(\phi(s)+1)y^{\beta}(s)\mbox{d}s].
\end{align*}
Again applying Gronwall's inequality, we have
\begin{align*}
y(t)\leq C, \ \forall\  t\in [0,T],
\end{align*}
which implies
\begin{align*}
\|\rho\|_{2m\beta+1}(t)\leq C m^{\frac{2}{\beta-1}},\ \forall \ t\in [0,T].
\end{align*}
\end{proof}
Step 3. First-order derivative estimates of the velocity $u$ and the magnetic field $H$.

\begin{lem}\label{lem3.5}
There exists a positive constant $C$, such that
\begin{align*}
\sup_{t\in[0,T]}\int\mu \omega^2+(\mbox{curl H})^2+\frac{(F+\frac12|H|^2)^2}{2\mu+\lambda(\rho)}\mbox{d}x+\int_0^T\int \rho(B^2+L^2)+\nu |\nabla \mbox{curl}H|^2\mbox{d}x\mbox{d}t\leq C.
\end{align*}
\end{lem}
\begin{proof}
Multiplying the first equation and the second equation of \eqref{omega and f} by $\mu \omega$ and $\frac{(F+\frac12|H|^2)}{2\mu+\lambda(\rho)}$ respectively, and then summing the resulted equations together, one has
\begin{align}\label{wf1}
\begin{split}
&\frac12\frac{\mbox{d}}{\mbox{d}t}\int\mu\omega
+\frac{(F+\frac12|H|^2)^2}{2\mu+\lambda(\rho)}\mbox{d}x
+\int \rho(B^2+L^2)\mbox{d}x\\&=-\frac{\mu}{2}\int\omega^2\mbox{div}u\mbox{d}x
-(\frac{1}{\rho}H\cdot\nabla H_1)_{x_2}\cdot\omega+(\frac{1}{\rho}H\cdot\nabla H_1)_{x_1}\cdot\omega\mbox{d}x\\
&+\frac12\int(F+\frac12|H|^2)^2\mbox{div}u[\rho(\frac{1}{2\mu+\lambda(\rho)})'-\frac{1}{2\mu+\lambda(\rho)}]\mbox{d}x\\
&+\int(F+\frac12|H|^2)\mbox{div}u[\rho(\frac{p}{2\mu+\lambda(\rho)})'-\frac{p}{2\mu+\lambda(\rho)}]\mbox{d}x\\
&-\int2(F+\frac12|H|^2)(u_{1x_2}u_{2x_1}-u_{1x_1}u_{2x_2})\mbox{d}x\\
&-\int\frac{2}{\rho}(F+\frac12|H|^2)(H_{1x_2}H_{2x_1}-H_{1x_1}H_{2x_2})\mbox{d}x\\
&+\int(F+\frac12|H|^2)[(\frac{1}{\rho})_{x_1} H \cdot \nabla H_1+(\frac{1}{\rho})_{x_2}H\cdot \nabla H_2]\mbox{d}x\\
&+\int B(\frac12|H|^2)_{x_2}+L(\frac12|H|^2)_{x_1}\mbox{d}x,
\end{split}
\end{align}
where we have used the fact that
\begin{align*}
&(u_{1x_1})^2+2u_{1x_2}u_{2x_1}+(u_{2x_2})^2\\&=(u_{1x_1}+u_{2x_2})^2+2(u_{1x_2}u_{2x_1}-u_{1x_1}u_{2x_2})
\\&=(\mbox{div}u)^2+2(u_{1x_2}u_{2x_1}-u_{1x_1}u_{2x_2})\\&=\mbox{div}u(\frac{F+\frac12|H|^2+p(\rho)}{2\mu+\lambda(\rho)})
+2(u_{1x_2}u_{2x_1}-u_{1x_1}u_{2x_2})
\end{align*}
and
\begin{align*}
&(H_{1x_1})^2+2H_{1x_2}H_{2x_1}+(H_{2x_2})^2\\&=(H_{1x_1}+H_{2x_2})^2+2(H_{1x_2}H_{2x_1}-H_{1x_1}H_{2x_2})
\\&=2(H_{1x_2}H_{2x_1}-H_{1x_1}H_{2x_2}).
\end{align*}
Applying the operator $\mbox{curl}$ to the magnetic equation of \eqref{BL}, multiplying the resulting equation by $\mbox{curl}H$, then integrating it over $\mathbb{T}^2$, we have
\begin{align*}
&\frac12\frac{\mbox{d}}{\mbox{d}t}(\mbox{curl}H)^2\mbox{d}x+\nu\int|\nabla\mbox{curl}|^2\mbox{d}x
-\frac12\int(\mbox{curl}H)^2\mbox{div}u\mbox{d}x\\
&+\int u_{x_1}\cdot\nabla H_{x_2}\cdot\mbox{curl}H\mbox{d}x-\int u_{x_2}\cdot\nabla H_{x_1}\cdot\mbox{curl}H\mbox{d}x\\
&+\int H\cdot\nabla u_2+H_2\mbox{div}u(\mbox{curl}H)_{x_1}\mbox{d}x-\int H\cdot\nabla u_1+H_1\mbox{div}u(\mbox{curl}H)_{x_2}\mbox{d}x,
\end{align*}
which together with \eqref{wf1} gives that
\begin{align}\label{wf2}
\begin{split}
&\frac12\frac{\mbox{d}}{\mbox{d}t}\int\mu\omega
+(\mbox{curl}H)^2+\frac{(F+\frac12|H|^2)^2}{2\mu+\lambda(\rho)}\mbox{d}x
+\int \rho(B^2+L^2)+\nu|\nabla\mbox{curl}|^2\mbox{d}x\\&=-\frac{\mu}{2}\int\omega^2\mbox{div}u\mbox{d}x
-(\frac{1}{\rho}H\cdot\nabla H_1)_{x_2}\cdot\omega+(\frac{1}{\rho}H\cdot\nabla H_1)_{x_1}\cdot\omega\mbox{d}x\\
&+\frac12\int(F+\frac12|H|^2)^2\mbox{div}u[\rho(\frac{1}{2\mu+\lambda(\rho)})'-\frac{1}{2\mu+\lambda(\rho)}]\mbox{d}x\\
&+\int(F+\frac12|H|^2)\mbox{div}u[\rho(\frac{p}{2\mu+\lambda(\rho)})'-\frac{p}{2\mu+\lambda(\rho)}]\mbox{d}x\\
&-\int2(F+\frac12|H|^2)(u_{1x_2}u_{2x_1}-u_{1x_1}u_{2x_2})\mbox{d}x\\
&-\int\frac{2}{\rho}(F+\frac12|H|^2)(H_{1x_2}H_{2x_1}-H_{1x_1}H_{2x_2})\mbox{d}x\\
&+\int(F+\frac12|H|^2)[(\frac{1}{\rho})_{x_1} H \cdot \nabla H_1+(\frac{1}{\rho})_{x_2}H\cdot \nabla H_2]\mbox{d}x\\
&+\int B(\frac12|H|^2)_{x_2}+L(\frac12|H|^2)_{x_1}\mbox{d}x-\frac12\int(\mbox{curl}H)^2\mbox{div}u\mbox{d}x\\
&-\int u_{x_1}\cdot\nabla H_{x_2}\cdot\mbox{curl}H\mbox{d}x+\int u_{x_2}\cdot\nabla H_{x_1}\cdot\mbox{curl}H\mbox{d}x\\
&-\int H\cdot\nabla u_2+H_2\mbox{div}u(\mbox{curl}H)_{x_1}\mbox{d}x+\int H\cdot\nabla u_1+H_1\mbox{div}u(\mbox{curl}H)_{x_2}\mbox{d}x.
\end{split}
\end{align}
Denote
\begin{align}\label{wf3}
Z^2(t)=\int\mu\omega
+(\mbox{curl}H)^2+\frac{(F+\frac12|H|^2)^2}{2\mu+\lambda(\rho)}\mbox{d}x
\end{align}
and
\begin{align}\label{wf4}
\varphi^2(t)=\int \rho(B^2+L^2)+\nu|\nabla\mbox{curl}|^2\mbox{d}x.
\end{align}
Then it holds that for $0<r\leq\frac12$,
\begin{align}\label{wf5}
\begin{split}
\|\nabla (F+|H|^2,\omega, \mbox{curl}H)\|_{2(1-r)}&\leq C\varphi(t)\|\rho\|^{\frac12}_{\frac{1-r}{r}}\leq C\varphi(t)(\frac{1-r}{r})^{\frac{1}{\beta-1}}\\&\leq C\varphi(t)r^{\frac{1}{1-\beta}}
\end{split}
\end{align}
and
\begin{align}\label{wf6}
\begin{split}
\|\nabla u\|_{2}+\|\omega\|_{2}+\|(2\mu+\lambda(\rho))^{\frac12}\mbox{div} u\|_{2}&\leq C[Z(t)+(\int\frac{p(\rho)^2}{2\mu+\lambda(\rho)}\mbox{d}x)^{\frac12}]\\&\leq C(Z(t)+1).
\end{split}
\end{align}

Now let's estimates the terms on the right-hand side of \eqref{wf2}. From the H\"{o}lder inequality, interpolation inequality, the elementary estimate \eqref{elementary estimate}, \eqref{wf5}-\eqref{wf6}, it follows that for $0<\varepsilon\leq \frac14$,
\begin{align}\label{wf7}
\begin{split}
-\frac{\mu}{2}\int\omega^2\mbox{div}u\mbox{d}x&\leq C\|\mbox{div}u\|_2\|\omega\|_4^2\leq C(Z(t)+1)\|\omega\|_2^{\frac{1-3\varepsilon}{1-2\varepsilon}}
\|\nabla\omega\|_{2(1-\varepsilon)}^{\frac{1-\varepsilon}{1-2\varepsilon}}\\
&\leq C(Z(t)+1)Z(t)^{\frac{1-3\varepsilon}{1-2\varepsilon}}\varphi(t)^{\frac{1-\varepsilon}{1-2\varepsilon}}
\varepsilon^{\frac{1-\varepsilon}{(1-\beta)(1-2\varepsilon)}}\\
&\leq \delta \varphi^2(t)+C_{\delta}Z^2(t)(Z(t)+1)^{\frac{2(1-2\varepsilon)}{1-3\varepsilon}}
\varepsilon^{\frac{2(1-\varepsilon)}{(1-\beta)(1-3\varepsilon)}}\\
&\leq \delta \varphi^2(t)+C_{\delta}(Z
^2(t)+1)^{2+\frac{\varepsilon}{1-3\varepsilon}}
\varepsilon^{\frac{2(1-\varepsilon)}{(1-\beta)(1-3\varepsilon)}},
\end{split}
\end{align}
where and in the sequel $\delta>0$ is a small positive constant to be determined and $C_{\delta}$ is a positive constant depending on $\delta$.

From the definition of $F$ and $\lambda(\rho)$, and Lemma \ref{lem_phase}, similarly one has
\begin{align}\label{wf8}
\begin{split}
&-(\frac{1}{\rho}H\cdot\nabla H_1)_{x_2}\cdot\omega+(\frac{1}{\rho}H\cdot\nabla H_1)_{x_1}\cdot\omega\mbox{d}x\leq C\|\nabla \omega\|_2\|\nabla H\|_2\|H\|_{L^{\infty}}\\
&\leq C\varphi(t)\|\mbox{curl}H\|_2
\leq \delta \varphi^2(t)+C_{\delta}(Z(t)^2+1),
\end{split}
\end{align}
\begin{align}\label{wf8}
\begin{split}
&\frac12\int(F+\frac12|H|^2)^2\mbox{div}u[\rho(\frac{1}{2\mu+\lambda(\rho)})'-\frac{1}{2\mu+\lambda(\rho)}]\mbox{d}x\\
&=\frac12\int(F+\frac12|H|^2)^2(\frac{F+\frac12|H|^2}{2\mu+\lambda(\rho)}+\frac{p(\rho)}{2\mu+\lambda(\rho)})
\frac{2\mu+\lambda(\rho)+\rho\lambda'(\rho)}{(2\mu+\lambda(\rho))^2}\mbox{d}x\\
& \leq C \int(F+\frac12|H|^2)^2(\frac{F+\frac12|H|^2}{2\mu+\lambda(\rho)}+\frac{p(\rho)}{2\mu+\lambda(\rho)})\mbox{d}x\\
& \leq C(1+\int\frac{\big|F+\frac12|H|^2\big|^3}{2\mu+\lambda(\rho)}\mbox{d}x),
\end{split}
\end{align}
\begin{align}\label{wf9}
\begin{split}
&\int(F+\frac12|H|^2)\mbox{div}u[\rho(\frac{p}{2\mu+\lambda(\rho)})'-\frac{p}{2\mu+\lambda(\rho)}]\mbox{d}x
\\&=\int(F+\frac12|H|^2)(\frac{F+\frac12|H|^2}{2\mu+\lambda(\rho)}+\frac{p(\rho)}{2\mu+\lambda(\rho)})\\
& \times\frac{p(\rho)(2\mu+\lambda(\rho))+\rho\lambda'(\rho)p(\rho)-\rho p'(\rho)(2\mu+\lambda(\rho))}{(2\mu+\lambda(\rho))^2}\mbox{d}x\\
& \leq C \int(F+\frac12|H|^2)(\frac{F+\frac12|H|^2}{2\mu+\lambda(\rho)}+\frac{p(\rho)}{2\mu+\lambda(\rho)})p(\rho)\mbox{d}x\\
& \leq C(1+\int\frac{\big|F+\frac12|H|^2\big|^3}{2\mu+\lambda(\rho)}\mbox{d}x),
\end{split}
\end{align}
\begin{align}\label{wf10}
-\int2(F+\frac12|H|^2)(u_{1x_2}u_{2x_1}-u_{1x_1}u_{2x_2})\mbox{d}x\leq C \int \big|F+\frac12|H|^2\big||\nabla u|^2\mbox{d}x,
\end{align}
\begin{align}\label{wf11}
\begin{split}
&-\int\frac{2}{\rho}(F+\frac12|H|^2)(H_{1x_2}H_{2x_1}-H_{1x_1}H_{2x_2})\mbox{d}x\\
&\leq C\int \big|F+\frac12|H|^2\big||\nabla H|^2\mbox{d}x,
\end{split}
\end{align}
\end{proof}
\begin{align}\label{wf12}
\begin{split}
&\int B(\frac12|H|^2)_{x_2}+L(\frac12|H|^2)_{x_1}\mbox{d}x\\
&\leq C(\|\nabla H\|_2\|H\|_{L^{\infty}}\|\sqrt{\rho}B\|_2+\|\nabla H\|_2\|H\|_{L^{\infty}}\|\sqrt{\rho}L\|_2)\\
&\leq C(Z(t)+1)\varphi(t)\leq \delta\varphi^2(t)+C_{\delta}(Z^2(t)+1),
\end{split}
\end{align}
\begin{align}\label{wf13}
\begin{split}
&-\frac12\int(\mbox{curl}H)^2\mbox{div}u\mbox{d}x\leq C\|\mbox{div} u\|_2\|\mbox{curl}H\|_{4}^2\\
&\leq C(Z(t)+1)\|\mbox{curl}H\|_{2}^{\frac{1-3\varepsilon}{1-2\varepsilon}}
\|\nabla\mbox{curl}H\|_{2(1-\varepsilon)}^{\frac{1-\varepsilon}{1-2\varepsilon}}\\
&\leq C(Z(t)+1)Z(t)^{\frac{1-3\varepsilon}{1-2\varepsilon}}\varphi(t)^{\frac{1-\varepsilon}{1-2\varepsilon}}
\varepsilon^{\frac{1-\varepsilon}{(1-\beta)(1-2\varepsilon)}}\\
&\leq \delta \varphi^2(t)+C_{\delta}Z^2(t)(Z(t)+1)^{\frac{2(1-2\varepsilon)}{1-3\varepsilon}}
\varepsilon^{\frac{2(1-\varepsilon)}{(1-\beta)(1-3\varepsilon)}}\\
&\leq \delta \varphi^2(t)+C_{\delta}(Z(t)^2+1)^{2+\frac{\varepsilon}{1-3\varepsilon}}
\varepsilon^{\frac{2(1-\varepsilon)}{(1-\beta)(1-3\varepsilon)}},
\end{split}
\end{align}
\begin{align}\label{wf14}
\begin{split}
&-\int u_{x_1}\cdot\nabla H_{x_2}\cdot\mbox{curl}H\mbox{d}x+\int u_{x_2}\cdot\nabla H_{x_1}\cdot\mbox{curl}H\mbox{d}x\\
&\leq C\|\nabla H\|_4^2\|\nabla u\|_2\leq C(Z(t)+1)\|\mbox{curl}H\|_4^2\\
&\leq C(Z(t)+1)\|\mbox{curl}H\|_2^{\frac{1-3\varepsilon}{1-2\varepsilon}}
\|\nabla\mbox{curl}H\|_{2(1-\varepsilon)}^{\frac{1-\varepsilon}{1-2\varepsilon}}\\
&\leq \delta \varphi^2(t)+C_{\delta}(Z(t)^2+1)^{2+\frac{\varepsilon}{1-3\varepsilon}}
\varepsilon^{\frac{2(1-\varepsilon)}{(1-\beta)(1-3\varepsilon)}}
\end{split}
\end{align}
and
\begin{align}\label{wf15}
\begin{split}
&-\int H\cdot\nabla u_2+H_2\mbox{div}u(\mbox{curl}H)_{x_1}\mbox{d}x+\int H\cdot\nabla u_1+H_1\mbox{div}u(\mbox{curl}H)_{x_2}\mbox{d}x\\
&\leq \|\nabla\mbox{curl}H\|_2\|\nabla u\|_2\|H\|_{L^{\infty}}
\leq C\varphi(t)(Z(t)+1)\\
&\leq \delta\varphi^2(t)+C_{\delta}(Z^2(t)+1).
\end{split}
\end{align}
Substituting \eqref{wf8}-\eqref{wf15} into \eqref{wf7} leads to \begin{align}\label{wf16}
\begin{split}
\frac12\frac{\mbox{d}}{\mbox{d}t}Z^2(t)+\varphi^2(t)&\leq \delta \varphi^2(t)+C_{\delta}\big(1
+Z^2(t)+(Z^2(t)+1)^{2+\frac{\varepsilon}{1-3\varepsilon}}
\varepsilon^{\frac{2(1-\varepsilon)}{(1-\beta)(1-3\varepsilon)}}\big)\\
&+C\big[1+\int\frac{\big|F+\frac12|H|^2\big|^3}{2\mu+\lambda(\rho)}\mbox{d}x+\int\big|F+\frac12|H|^2\big||\nabla u|^2\mbox{d}x\\
&+\int\big|F+\frac12|H|^2\big||\nabla H|^2\mbox{d}x\big].
\end{split}
\end{align}
Now it remains to estimate the terms $\int\frac{\big|F+\frac12|H|^2\big|^3}{2\mu+\lambda(\rho)}\mbox{d}x$, $\int\big|F+\frac12|H|^2\big||\nabla u|^2\mbox{d}x$ and $\int\big|F+\frac12|H|^2\big||\nabla H|^2\mbox{d}x$ on the right-hand side of \eqref{wf16}.

From Lemma \ref{lem3}, for $\varepsilon\in[0,\frac12]$ and $\eta=\varepsilon$, it follows that
\begin{align}\label{wf17}
\|F+\frac12|H|^2\|_{2m}\leq C[\|F+\frac12|H|^2\|_{1}+m^{\frac12}\|\nabla(F+\frac12|H|^2)\|_{\frac{2m}{m+\varepsilon}}^{1-s}
\|F+\frac12|H|^2\|_{2(1-\varepsilon)}^s],
\end{align}
where $s=\frac{(1-\varepsilon)^2}{m-\varepsilon(1-\varepsilon)}$ and the positive constant $C$ is independent of $m$ and $\varepsilon$.
Choose the positive constant $\varepsilon=2^{-m}$ with $m>2$ being integer in the inequalities \eqref{wf16} and \eqref{wf17}. From Lemma \ref{lem3.4} one can get
\begin{align}\label{wf18}
\begin{split}
&\|F+\frac12|H|^2\|_{1}=\int(2\mu+\lambda(\rho))^{-\frac12}|F+\frac12|H|^2|(2\mu+\lambda(\rho))^{\frac12}\mbox{d}x\\
&\leq (\int\frac{(F+\frac12|H|^2)^2}{2\mu+\lambda(\rho)}\mbox{d}x)^{\frac12}(\int2\mu+\lambda(\rho)\mbox{d}x)^{\frac12}\leq C Z(t),
\end{split}
\end{align}
\begin{align}\label{wf19}
\begin{split}
&\|F+\frac12|H|^2\|_{2(1-\varepsilon)}^s\\
&=\big(\int(2\mu+\lambda(\rho))^{-(1-\varepsilon)}
|F+\frac12|H|^2|^{2(1-\varepsilon)}(2\mu+\lambda(\rho))^{1-\varepsilon}\mbox{d}x\big)^{\frac{s}{2(1-\varepsilon)}}\\
&\leq \big(\int\frac{(F+\frac12|H|^2)^2}{2\mu+\lambda(\rho)}\mbox{d}x\big)^{\frac s2}\big(\int(2\mu+\lambda(\rho))^{\frac{1-\varepsilon}{\varepsilon}}\mbox{d}x\big)^{\frac{s\varepsilon}{2(1-\varepsilon)}}\\
&\leq C Z^s(t)(\|\rho\|_{\frac{\beta(1-\varepsilon)}{\varepsilon}}^{\frac{s\beta}{2}}+1)\leq CZ^s(t)[(\frac{\beta(1-\varepsilon)}{\varepsilon})^{\frac{s\beta}{\beta-1}}+1]\\
&\leq CZ^s(t)[(\varepsilon^{-\frac{s\beta}{\beta-1}}+1]\leq C Z^s(t)(2^{\frac{ms\beta}{\beta-1}}+1)\leq CZ^s(t),
\end{split}
\end{align}
where in the last inequality one has used the fact that $ms=\frac{m(1-\varepsilon)^2}{m-\varepsilon(1-\varepsilon)}\rightarrow1$, as $m\rightarrow\infty$.

Inserting \eqref{wf5} with $r=\frac{\varepsilon}{m+\varepsilon}$, \eqref{wf18}-\eqref{wf19} into \eqref{wf17} implies that
\begin{align}\label{wf20}
\begin{split}
&\|F+\frac12|H|^2\|_{2m}\leq C[Z(t)+m^{\frac12}\|\nabla(F+\frac12|H|^2)\|_{\frac{2m}{m+\varepsilon}}^{1-s}
Z(t)^s]\\
&\leq C[Z(t)+m^{\frac12}(\frac{m+\varepsilon}{\varepsilon})^{\frac{1-s}{\beta-1}}\varphi(t)^{1-s}Z(t)^s]\\
&\leq C[Z(t)+m^{\frac12}(\frac{m}{\varepsilon})^{\frac{1-s}{\beta-1}}\varphi(t)^{1-s}Z(t)^s].
\end{split}
\end{align}
Therefore, it holds that
\begin{align}\label{wf3.54}
\begin{split}
&\int\frac{\big|F+\frac12|H|^2\big|^3}{2\mu+\lambda(\rho)}\mbox{d}x
\\
&=\int\frac{\big|F+\frac12|H|^2\big|^{2-\frac{1}{m-1}}}{(2\mu+\lambda(\rho))^{1-\frac{1}{2(m-1)}}}
(\frac{1}{2\mu+\lambda(\rho)})^{\frac{1}{2(m-1)}}\big|F+\frac12|H|^2\big|^{1+\frac{1}{m-1}}\mbox{d}x\\
&\leq \int\big(\frac{\big|F+\frac12|H|^2\big|^2}{2\mu+\lambda(\rho)}\big)^{1-\frac{1}{2(m-1)}}\big|F+\frac12|H|^2\big|^{\frac{m}{m-1}}\mbox{d}x\\
&\leq \big(\int\frac{\big|F+\frac12|H|^2\big|^2}{2\mu+\lambda(\rho)}\mbox{d}x\big)^{\frac{2m-3}{2(m-1)}}
\big(\int\big|F+\frac12|H|^2\big|^{2m}\mbox{d}x\big)^{\frac{1}{2(m-1)}}\\
&\leq Z(t)^{\frac{2m-3}{m-1}}\|F+\frac12|H|^2\|_{2m}^{\frac{m}{m-1}}\\
&\leq CZ(t)^{\frac{2m-3}{m-1}}[Z(t)+m^{\frac12}(\frac{m}{\varepsilon})^{\frac{1-s}{\beta-1}}\varphi(t)^{1-s}Z(t)^s]^{\frac{m}{m-1}}\\
&\leq C[Z(t)^3+m^{\frac{m}{2(m-1)}}(\frac{m}{\varepsilon})^{\frac{(1-s)m}{(\beta-1)(m-1)}}\varphi(t)^{\frac{(1-s)m}{m-1}}Z(t)^{\frac{(2+s)m-3}{m-1}}]\\
&\leq \delta \varphi^2(t)+C_{\delta}[Z(t)^3+m^{\frac{m}{m(1+s)-2)}}
(\frac{m}{\varepsilon})^{\frac{2(1-s)m}{(\beta-1)(m(1+s)-2}}Z(t)^{\frac{2((2+s)m-3)}{m(1+s)-2}}]\\
&\leq \delta \varphi^2(t)+C_{\delta}[(1+Z^2(t))^2+m(\frac{m}{\varepsilon})^{\frac{2}{\beta-1}}(1+Z^2(t))^{2+\frac{1-ms}{m(1+s)-2}}],
\end{split}
\end{align}
where we have applied the fact $$ms=\frac{m(1-\varepsilon)^2}{m-\varepsilon(1-\varepsilon)}\rightarrow1\ \mbox{as}\ m\rightarrow+\infty$$ and $$\lim_{m\rightarrow+\infty}(2^m(1-ms))=2,$$ with $\varepsilon=2^{-m}$.

Now consider the integral
\begin{align}\label{wf21}
\begin{split}
&\int\big|F+\frac12|H|^2\big||\nabla u|^2\mbox{d}x\leq \|F+\frac12|H|^2\|_{2m}\|\nabla u\|_{\frac{4m}{2m-1}}^2\\
&\leq C\|F+\frac12|H|^2\|_{2m}(\|\mbox{div} u\|_{\frac{4m}{2m-1}}^2+\|\omega\|_{\frac{4m}{2m-1}}^2)\\
&\leq C\|F+\frac12|H|^2\|_{2m}(\|\frac{F+\frac12|H|^2}{2\mu+\lambda(\rho)}\|_{\frac{4m}{2m-1}}^2+\|\omega\|_{\frac{4m}{2m-1}}^2+1).
\end{split}
\end{align}
Note that
\begin{align}\label{wf22}
\begin{split}
&\|\frac{F+\frac12|H|^2}{2\mu+\lambda(\rho)}\|_{\frac{4m}{2m-1}}^2
=\big(\int\frac{\big|F+\frac12|H|^2\big|^{\frac{4m}{2m-1}}}{(2\mu+\lambda(\rho))^{\frac{4m}{2m-1}}}\mbox{d}x\big)^{\frac{2m-1}{2m}}\\
&=\big(\int\frac{\big|F+\frac12|H|^2\big|^{\frac{2m(2m-3)}{(2m-1)(m-1)}}}{(2\mu+\lambda(\rho))^{\frac{4m}{2m-1}}}
\big|F+\frac12|H|^2\big|^{\frac{2m}{(2m-1)(m-1)}}\mbox{d}x\big)^{\frac{2m-1}{2m}}\\
&\leq \|F+\frac12|H|^2\|_{2m}^{\frac{1}{m-1}}\big(\int\frac{\big|F+\frac12|H|^2
\big|^2}{(2\mu+\lambda(\rho))^{\frac{4(m-1)}{2m-3}}}\mbox{d}x\big)^{\frac{2m-3}{2(m-1)}}\\
&\leq C \|F+\frac12|H|^2\|_{2m}^{\frac{1}{m-1}}\big(\int\frac{\big|F+\frac12|H|^2
\big|^2}{2\mu+\lambda(\rho)}\mbox{d}x\big)^{\frac{2m-3}{2(m-1)}}\\
&\leq C\|F+\frac12|H|^2\|_{2m}^{\frac{1}{m-1}}Z(t)^{\frac{2m-3}{2(m-1)}}
\end{split}
\end{align}
and by virtue of $\int\omega\mbox{d}x=0$, interpolation inequality and \eqref{wf5}, it can be derived as
\begin{align}
\begin{split}
\|\omega\|_{\frac{4m}{2m-1}}^2&\leq C\|\omega\|^{2-\frac{1-\varepsilon}{m(1-2\varepsilon)}}_2
\|\nabla\omega\|^{\frac{1-\varepsilon}{m(1-2\varepsilon)}}_{2(1-\varepsilon)}\\
&\leq CZ(t)^{2-\frac{1-\varepsilon}{m(1-2\varepsilon)}}
(\varphi(t)\varepsilon^{\frac{1}{1-\beta}})^{\frac{1-\varepsilon}{m(1-2\varepsilon)}}\\
&\leq CZ(t)^{2-\frac{1-\varepsilon}{m(1-2\varepsilon)}}2^{\frac{1-\varepsilon}{(\beta-1)(1-2\varepsilon)}}
\varphi(t)^{\frac{1-\varepsilon}{m(1-2\varepsilon)}}\\
&\leq CZ(t)^{2-\frac{1-\varepsilon}{m(1-2\varepsilon)}}\varphi(t)^{\frac{1-\varepsilon}{m(1-2\varepsilon)}},
\end{split}
\end{align}
which together with \eqref{wf20},\eqref{wf21}, \eqref{wf22} implies that
\begin{align}\label{velocityintegral}
\begin{split}
&\int\big|F+\frac12|H|^2\big||\nabla u|^2\mbox{d}x\\
&\leq C[Z(t)+m^{\frac12}(\frac{m}{\varepsilon})^{\frac{1-s}{\beta-1}}\varphi(t)^{1-s}Z(t)^s]^{1+\frac{1}{m-1}}
Z(t)^{\frac{2m-3}{2(m-1)}}\\
&+C [Z(t)+m^{\frac12}(\frac{m}{\varepsilon})^{\frac{1-s}{\beta-1}}\varphi(t)^{1-s}Z(t)^s]
[1+Z(t)^{2-\frac{1-\varepsilon}{m(1-2\varepsilon)}}\varphi(t)^{\frac{1-\varepsilon}{m(1-2\varepsilon)}}]\\
&\leq C[Z(t)^3+m^{\frac{m}{2(m-1)}}(\frac{m}{\varepsilon})^{\frac{m(1-s)}{(\beta-1)(m-1)}}
\varphi(t)^{\frac{m(1-s)}{m-1}}Z(t)^{\frac{ms+2m-3}{m-1}}+Z(t)\\
&
+Z(t)^{3-\frac{1-\varepsilon}{m(1-2\varepsilon)}}\varphi(t)^{\frac{1-\varepsilon}{m(1-2\varepsilon)}}
+m^{\frac12}(\frac{m}{\varepsilon})^{\frac{1-s}{\beta-1}}\varphi(t)^{1-s}Z(t)^s\\
&+m^{\frac12}(\frac{m}{\varepsilon})^{\frac{1-s}{\beta-1}}\varphi(t)^{1-s+\frac{1-\varepsilon}{m(1-2\varepsilon)}}
Z(t)^{2+s-\frac{1-\varepsilon}{m(1-2\varepsilon)}}]\\
&\leq \delta \varphi^2(t)+C_{\delta}[(1+Z^2(t))^2+\big(m^{\frac12}(\frac{m}{\varepsilon})^{\frac{1-s}{\beta-1}}Z^s(t)\big)^{\frac{2}{1+s}}
\\&+\big(m^{\frac{m}{2(m-1)}}(\frac{m}{\varepsilon})^{\frac{m(1-s)}{(\beta-1)(m-1)}}
Z(t)^{2+\frac{ms-1}{m-1}}\big)^{\frac{2(m-1)}{m(1+s)-2}}\\
&+\big(m^{\frac12}(\frac{m}{\varepsilon})^{\frac{1-s}{\beta-1}}
Z(t)^{2+s-\frac{1-\varepsilon}{m(1-2\varepsilon)}}\big)^{\frac{2}{1+s-\frac{1-\varepsilon}{m(1-2\varepsilon)}}}]\\
&\leq \delta \varphi^2(t)+C_{\delta}[(1+Z^2(t))^2+m(\frac{m}{\varepsilon})^{\frac{2}{\beta-1}}(1+Z^2(t))^{2+\frac{1-ms}{m(1+s)-2}}\\
&+m(\frac{m}{\varepsilon})^{\frac{2}{\beta-1}}(1+Z^2(t))+m(\frac{m}{\varepsilon})^{\frac{2}{\beta-1}}
(1+Z^2(t))^{2+\frac{1-ms+(2ms-1)\varepsilon}{m(1+s)(1-2\varepsilon)-1+\varepsilon}}].
\end{split}
\end{align}
By $\int \mbox{curl}H\mbox{d}x=0$, Lemma \ref{lem2}, \eqref{wf5} and \eqref{wf20}, after a similar estimate we arrive at
\begin{align}\label{magneticintegral}
\begin{split}
&\int\big|F+\frac12|H|^2\big||\nabla H|^2\mbox{d}x\\&
\leq \|F+\frac12|H|^2\|_{2m}\|\nabla H\|_{\frac{4m}{2m-1}}^2
\leq C\|F+\frac12|H|^2\|_{2m}\|\mbox{curl}H\|_{\frac{4m}{2m-1}}^2\\
&\leq C\|F+\frac12|H|^2\|_{2m}\|\mbox{curl}H\|^{2-\frac{1-\varepsilon}{m(1-2\varepsilon)}}_2
\|\nabla\mbox{curl}H\|^{\frac{1-\varepsilon}{m(1-2\varepsilon)}}_{2(1-\varepsilon)}\\
&\leq C [Z(t)+m^{\frac12}(\frac{m}{\varepsilon})^{\frac{1-s}{\beta-1}}\varphi(t)^{1-s}Z(t)^s]
Z(t)^{2-\frac{1-\varepsilon}{m(1-2\varepsilon)}}\varphi(t)^{\frac{1-\varepsilon}{m(1-2\varepsilon)}}\\
&\leq
C[m^{\frac12}(\frac{m}{\varepsilon})^{\frac{1-s}{\beta-1}}\varphi(t)^{1-s+\frac{1-\varepsilon}{m(1-2\varepsilon)}}
Z(t)^{2+s-\frac{1-\varepsilon}{m(1-2\varepsilon)}}\\
&+Z(t)^{3-\frac{1-\varepsilon}{m(1-2\varepsilon)}}\varphi(t)^{\frac{1-\varepsilon}{m(1-2\varepsilon)}}]\\
&\leq \delta \varphi^2(t)+C_{\delta}[(1+Z^2(t))^2+\big(m^{\frac12}(\frac{m}{\varepsilon})^{\frac{1-s}{\beta-1}}Z^s(t)\big)^{\frac{2}{1+s}}
\\&+\big(m^{\frac{m}{2(m-1)}}(\frac{m}{\varepsilon})^{\frac{m(1-s)}{(\beta-1)(m-1)}}
Z(t)^{2+\frac{ms-1}{m-1}}\big)^{\frac{2(m-1)}{m(1+s)-2}}\\
&+\big(m^{\frac12}(\frac{m}{\varepsilon})^{\frac{1-s}{\beta-1}}
Z(t)^{2+s-\frac{1-\varepsilon}{m(1-2\varepsilon)}}\big)^{\frac{2}{1+s-\frac{1-\varepsilon}{m(1-2\varepsilon)}}}]\\
&\leq \delta \varphi^2(t)+C_{\delta}[(1+Z^2(t))^2+m(\frac{m}{\varepsilon})^{\frac{2}{\beta-1}}
(1+Z^2(t))^{2+\frac{1-ms+(2ms-1)\varepsilon}{m(1+s)(1-2\varepsilon)-1+\varepsilon}}].
\end{split}
\end{align}
Substituting \eqref{wf3.54}, \eqref{velocityintegral} and \eqref{magneticintegral} into \eqref{wf16} and choosing $\delta$ sufficiently small gives that
\begin{align}\label{wf16j}
\begin{split}
\frac12\frac{\mbox{d}}{\mbox{d}t}Z^2(t)+\frac12\varphi^2(t)&\leq C\big[+(1+Z^2(t))^2+(Z^2(t)+1)^{2+\frac{\varepsilon}{1-3\varepsilon}}
\varepsilon^{\frac{2(1-\varepsilon)}{(1-\beta)(1-3\varepsilon)}}
\\
&+m(\frac{m}{\varepsilon})^{\frac{2}{\beta-1}}(1+Z^2(t))^{2+\frac{1-ms}{m(1+s)-2}}
+m(\frac{m}{\varepsilon})^{\frac{2}{\beta-1}}(1+Z^2(t))\\
&+m(\frac{m}{\varepsilon})^{\frac{2}{\beta-1}}
(1+Z^2(t))^{2+\frac{1-ms+(2ms-1)\varepsilon}{m(1+s)(1-2\varepsilon)-1+\varepsilon}}\big].
\end{split}
\end{align}
Notice that
$$\lim_{m\rightarrow+\infty}[2^m(1-ms)]=2,$$ hence $1-ms\sim 2\varepsilon$ as $m\rightarrow +\infty$.
For $m$ sufficiently large enough, one can show that
$$\frac{1-ms}{m(1+s)-2}\leq 4\varepsilon,$$
and
$$\frac{1-ms+(2ms-1)\varepsilon}{m(1+s)(1-2\varepsilon)-1+\varepsilon}\leq 4\varepsilon.$$
Therefore, from \eqref{wf16j}, we obtain the following inequality
\begin{align}\label{wf16jj}
\begin{split}
\frac12\frac{\mbox{d}}{\mbox{d}t}Z^2(t)+\frac12\varphi^2(t)&\leq Cm(\frac{m}{\varepsilon})^{\frac{2}{\beta-1}}
(1+Z^2(t))^{2+4\varepsilon}.
\end{split}
\end{align}
On the other hand, from
 \begin{align*}
Z^2(t)=\int\mu\omega
+(\mbox{curl}H)^2+\frac{(F+\frac12|H|^2)^2}{2\mu+\lambda(\rho)}\mbox{d}x
\end{align*}
and the estimate for density and the elementary estimates for the velocity and magnetic field, it is easy to show that $Z^2(t)\in L^1(0,T)$. Thus, from \eqref{wf16jj}, we conclude
\begin{align*}
\frac{1}{(1+Z^2(t))^{4\varepsilon}}-\frac{1}{(1+Z^2(0))^{4\varepsilon}}+C m\varepsilon(\frac{m}{\varepsilon})^{\frac{2}{\beta-1}}\geq0.
\end{align*}
Take $M>2$ so as to satisfy the inequality
\begin{align*}
C M\varepsilon(\frac{M}{\varepsilon})^{\frac{2}{\beta-1}}\leq\frac{1}{2(1+Z^2(0))^{4\varepsilon}},
\end{align*}
that is,
\begin{align}\label{wf3.65}
C M^{1+\frac{2}{\beta-1}}2^{-m(1-\frac{2}{\beta-1})}\leq\frac{1}{2(1+Z^2(0))^{4\varepsilon}},
\end{align}
then
\begin{align}\label{wf3.63}
\frac{1}{(1+Z^2(t))^{4\varepsilon}}\geq\frac{1}{(1+Z^2(0))^{4\varepsilon}}.
\end{align}
Since
\begin{align*}
&Z^2(0)=\int\mu\omega_0
+(\mbox{curl}H_0)^2+\frac{(F_0+\frac12|H_0|^2)^2}{2\mu+\lambda(\rho_0)}\mbox{d}x\\
&\leq C[\|u_0\|_{H^2}^2+\|H_0\|_{H^2}^2+\|\rho_0\|_{H^3}^{\beta}\|u_0\|_{H^2}^2+\|\rho_0\|_{H^3}^{2\gamma}]\\
&\leq C,
\end{align*}
if $1-\frac{2}{\beta-1}>0$, that is $\beta>3$, then we can take sufficiently large $M>2$ to ensure the condition \eqref{wf3.65}. From \eqref{wf3.63}, it holds that
$$Z^2(t)\leq 2^{2^{m-2}}(1+Z^2(0))-1\leq C$$ and $$\int_0^T\varphi(t)\mbox{d}t\leq C.$$
Thus we complete the proof of Lemma \ref{lem3.5}.

Step 4. Second-order estimates for the velocity and the magnetic field
\begin{lem}\label{lem3.6}
There exists a positive constant $C$ independent of $\delta$,
such that
\begin{align*}
&\sup_{t\in[0,T]}\int \rho(B^2+L^2)+\nu|\nabla\mbox{curl}H|^2+|\nabla\rho|^2\mbox{d}x
+\int_0^T\int\mu(B_{x_1}-L_{x_2})^2\\
&+(2\mu+\lambda(\rho))(B_{x_2}+L_{x_1})^2\mbox{d}x\mbox{d}t\leq C.
\end{align*}
\end{lem}
\begin{proof}
Multiplying the first equation of the system \eqref{BL} by $B$, the second by $L$, and integrating their sum in the space variable over $\mathbb{T}^2$, then using the continuity equation, we get
\begin{align}\label{bl1}
\begin{split}
&\frac12\frac{\mbox{d}}{\mbox{d}t}\int\rho(B^2+L^2)\mbox{d}x+\int\mu(B_{x_1}-L_{x_2})^2
+(2\mu+\lambda(\rho))(B_{x_2}+L_{x_1})^2\mbox{d}x\\
&=\int\rho(B^2+L^2)\mbox{div}u\mbox{d}x
+\int\mu[(\frac{1}{2\rho}|H|^2_{x_2})_{x_1}-(\frac{1}{2\rho}|H|^2_{x_1})_{x_2}](B_{x_1}-L_{x_2})\mbox{d}x\\
&-\int\rho(2\mu+\lambda(\rho))
[(F+\frac{1}{2}|H|^2)(\frac{1}{2\mu+\lambda(\rho)})'
+(\frac{p}{2\mu+\lambda(\rho)})']\mbox{div}u(B_{x_2}+L_{x_2})\mbox{d}x\\&-\int B(u_{x_2}\cdot\nabla(F+\frac{1}{2}|H|^2)+\mu u_{x_1}\nabla\omega)+L(u_{x_1}\cdot\nabla(F+\frac{1}{2}|H|^2)-\mu u_{x_2}\nabla\omega)\mbox{d}x
\\&+\int\mu[\omega\mbox{div}u+(\frac{1}{\rho}H\cdot\nabla H_1)_{x_2}-(\frac{1}{\rho}H\cdot\nabla H_1)_{x_1}](B_{x_1}-L_{x_2})\mbox{d}x\\
&+\int(2\mu+\lambda(\rho))[(u_{1x_1})^2+2u_{1x_2}u_{2x_1}
+(u_{2x_2})^2](B_{x_2}+L_{x_1})\mbox{d}x\\&-\int\frac{2\mu+\lambda(\rho)}{\rho}
[(H_{1x_1})^2+2H_{1x_2}H_{2x_1}+(H_{2x_2})^2](B_{x_2}+L_{x_1})\mbox{d}x\\&
-\int(2\mu+\lambda(\rho)) [(\frac{1}{\rho})_{x_1} H \cdot \nabla H_1+(\frac{1}{\rho})_{x_2}H\cdot \nabla H_2](B_{x_2}+L_{x_2})\mbox{d}x
\\&+\int(2\mu+\lambda(\rho))(B_{x_2}+L_{x_1})[(\frac{1}{2\rho}|H|^2_{x_2})_{x_2}
+(\frac{1}{2\rho}|H|^2_{x_1})_{x_1}].
\end{split}
\end{align}
Applying the operator $\nabla\mbox{curl}$ to the magnetic equation, multiplying it by $\nabla\mbox{curl}H$, and then integrating over $\mathbb{T}^2$, we obtain
\begin{align}\label{bl2}
\begin{split}
&\frac12\frac{\mbox{d}}{\mbox{d}t}\|\nabla\mbox{curl}H\|_2^2+\nu\|\nabla^2\mbox{curl}H\|_2^2
\\&=-\int\nabla\mbox{curl}u\cdot\nabla H\cdot\nabla\mbox{curl}H
+\mbox{curl}u\cdot\nabla \nabla H\cdot\nabla\mbox{curl}H\\&+\nabla u\cdot\nabla\mbox{curl}H\cdot\nabla\mbox{curl}H+u\cdot\nabla\nabla\mbox{curl}H\cdot\nabla\mbox{curl}H
\\&-\nabla\mbox{curl}H\cdot\nabla u\cdot\nabla\mbox{curl}H-\mbox{curl}H\cdot\nabla\nabla u\cdot\nabla\mbox{curl}H\\
&-\nabla H\cdot \nabla\mbox{curl}u\cdot\nabla\mbox{curl}H-H\cdot\nabla\nabla\mbox{curl}u\cdot\nabla\mbox{curl}H\\
&+\mbox{div}u\nabla\mbox{curl}H\cdot\nabla\mbox{curl}H+\mbox{curl}H\cdot\nabla\mbox{div}u\cdot\nabla\mbox{curl}H\\
&+\nabla H\mbox{curl}\mbox{div}u\cdot\nabla\mbox{curl}H+H\cdot\nabla\mbox{curl}\mbox{div}u\cdot\nabla\mbox{curl}H\mbox{d}x.
\end{split}
\end{align}
Applying $\nabla$ to the mass equation, multiplying by $\nabla\rho$, then integrating the resulted equation by parts, one arrives at
\begin{align*}
\frac12\frac{\mbox{d}}{\mbox{d}t}\|\nabla\rho\|_2^2=-\int\nabla u|\nabla\rho|^2\mbox{d}x-\frac12\int\mbox{div}u|\nabla\rho|^2\mbox{d}x
-\int\rho\nabla\rho\cdot\nabla\mbox{div}u\mbox{d}x,
\end{align*}
which together with \eqref{bl1}-\eqref{bl2} gives that
\begin{align}\label{bl3}
\begin{split}
&\frac12\frac{\mbox{d}}{\mbox{d}t}\int\rho(B^2+L^2)+|\nabla\mbox{curl}H|^2+|\nabla\rho|^2\mbox{d}x\\
&+\int\mu(B_{x_1}-L_{x_2})^2
+(2\mu+\lambda(\rho))(B_{x_2}+L_{x_1})^2+\nu|\nabla^2\mbox{curl}H|^2\mbox{d}x\\
&=\int\rho(B^2+L^2)\mbox{div}u\mbox{d}x
+\int\mu[(\frac{1}{2\rho}|H|^2_{x_2})_{x_1}-(\frac{1}{2\rho}|H|^2_{x_1})_{x_2}](B_{x_1}-L_{x_2})\mbox{d}x\\
&-\int\rho(2\mu+\lambda(\rho))
[(F+\frac{1}{2}|H|^2)(\frac{1}{2\mu+\lambda(\rho)})'
+(\frac{p}{2\mu+\lambda(\rho)})']\mbox{div}u(B_{x_2}+L_{x_2})\mbox{d}x\\&-\int B(u_{x_2}\cdot\nabla(F+\frac{1}{2}|H|^2)+\mu u_{x_1}\nabla\omega)+L(u_{x_1}\cdot\nabla(F+\frac{1}{2}|H|^2)-\mu u_{x_2}\nabla\omega)\mbox{d}x
\\&+\int\mu[\omega\mbox{div}u+(\frac{1}{\rho}H\cdot\nabla H_1)_{x_2}-(\frac{1}{\rho}H\cdot\nabla H_1)_{x_1}](B_{x_1}-L_{x_2})\mbox{d}x\\
&+\int(2\mu+\lambda(\rho))[(u_{1x_1})^2+2u_{1x_2}u_{2x_1}
+(u_{2x_2})^2](B_{x_2}+L_{x_1})\mbox{d}x\\&-\int\frac{2\mu+\lambda(\rho)}{\rho}
[(H_{1x_1})^2+2H_{1x_2}H_{2x_1}+(H_{2x_2})^2](B_{x_2}+L_{x_1})\mbox{d}x\\&
-\int(2\mu+\lambda(\rho)) [(\frac{1}{\rho})_{x_1} H \cdot \nabla H_1+(\frac{1}{\rho})_{x_2}H\cdot \nabla H_2](B_{x_2}+L_{x_2})\mbox{d}x
\\&+\int(2\mu+\lambda(\rho))(B_{x_2}+L_{x_1})[(\frac{1}{2\rho}|H|^2_{x_2})_{x_2}
+(\frac{1}{2\rho}|H|^2_{x_1})_{x_1}]\\
&-\int\nabla\mbox{curl}u\cdot\nabla H\cdot\nabla\mbox{curl}H
+\mbox{curl}u\cdot\nabla \nabla H\cdot\nabla\mbox{curl}H+\nabla u\cdot\nabla\mbox{curl}H\cdot\nabla\mbox{curl}H\\&+u\cdot\nabla\nabla\mbox{curl}H\cdot\nabla\mbox{curl}H
-\nabla\mbox{curl}H\cdot\nabla u\cdot\nabla\mbox{curl}H-\mbox{curl}H\cdot\nabla\nabla u\cdot\nabla\mbox{curl}H\\
&-\nabla H\cdot \nabla\mbox{curl}u\cdot\nabla\mbox{curl}H-H\cdot\nabla\nabla\mbox{curl}u\cdot\nabla\mbox{curl}H
+\mbox{div}u\nabla\mbox{curl}H\cdot\nabla\mbox{curl}H\\&+\mbox{curl}H\cdot\nabla\mbox{div}u\cdot\nabla\mbox{curl}H
+\nabla H\mbox{curl}\mbox{div}u\cdot\nabla\mbox{curl}H+H\cdot\nabla\mbox{curl}\mbox{div}u\cdot\nabla\mbox{curl}H\mbox{d}x\\
&-\int\nabla u|\nabla\rho|^2\mbox{d}x-\frac12\int\mbox{div}u|\nabla\rho|^2\mbox{d}x
-\int\rho\nabla\rho\cdot\nabla\mbox{div}u\mbox{d}x.
\end{split}
\end{align}
Put
\begin{align}\label{bl4}
Y(t)=\big(\int\rho(B^2+L^2)+|\nabla\mbox{curl}H|^2+|\nabla\rho|^2\mbox{d}x\big)^{\frac{1}{2}}
\end{align}
and
\begin{align}\label{bl5}
\psi(t)=\big(\int\mu(B_{x_1}-L_{x_2})^2
+(2\mu+\lambda(\rho))(B_{x_2}+L_{x_1})^2+\nu|\nabla^2\mbox{curl}H|^2\mbox{d}x\big)^{\frac12}.
\end{align}
Note that
\begin{align*}
\int(|\nabla B|^2+|\nabla L|^2)\mbox{d}x&=\int(B_{x_1}^2+B_{x_2}^2+L_{x_1}^2+L_{x_2}^2)\mbox{d}x\\
&=\int[(B_{x_1}-L_{x_2})^2+(B_{x_2}+L_{x_1})]\mbox{d}x\\
&\leq \frac{1}{\mu}\psi^2(t).
\end{align*}
Thus it holds that
\begin{align}\label{bl6}
\|\nabla (B,L)\|_2(t)\leq C\psi(t), \ \forall t\in [0,T].
\end{align}
Then it follows from the elliptic system
$$\mu\omega_{x_1}+(F+\frac12|H|^2)_{x_2}=\rho B,$$
$$-\mu\omega_{x_2}+(F+\frac12|H|^2)_{x_1}=\rho L,$$
that
\begin{align}\label{bl7}
\|\nabla (F+\frac12|H|^2, \omega)\|_p\leq C\|\rho(B,L)\|_p,\ \forall 1<p<+\infty.
\end{align}
Furthermore, since $\int\mu\omega_{x_1}+(F+\frac12|H|^2)_{x_2}\mbox{d}x=0$, by the mean value theorem, there exists a point $x_*\in\mathbb{T}^2$, such that $(\omega_{x_1}+(F+\frac12|H|^2)_{x_2})({x_*},t)=0$, and so $B(x_*,t)=0$. Similarly, there exists a point $x^1_*$, such that $L(x_*^1,t)=0$. Therefore, by the Poincare inequality, it holds that
\begin{align}\label{bl3.74}
\|(B,L)\|_p\leq C\|\nabla(B,L)\|_2,\ \forall 1<p<+\infty ,
\end{align}
where $C$ depend on $p$.

Now we estimate the right-hand side of \eqref{bl3} term by term. From the H\"{o}lder inequality, \eqref{bl3.74} and Lemma \ref{lem3.4}, it holds that
\begin{align}\label{bl8}
\begin{split}
|\int\rho(B^2+L^2)\mbox{div}u\mbox{d}x|&=|\int\rho(B^2+L^2)\frac{F+\frac12|H|^2+p(\rho)}{2\mu+\lambda(\rho)}\mbox{d}x|\\
&\leq \|\sqrt{\rho}(B,L)\|_2\|(B,L)\|_4\|\frac{\sqrt{\rho}(F+\frac12|H|^2+p(\rho))}{2\mu+\lambda(\rho)}\|_4\\
&\leq CY(t)\psi(t)(1+\|F+\frac12|H|^2\|_4).
\end{split}
\end{align}
Observe that
\begin{align}\label{bl9}
\|(F+\frac12|H|^2,\omega)\|_4\leq C\|\nabla(F+\frac12|H|^2,\omega)\|_2\leq CY(t).
\end{align}
Then we can write \eqref{bl8} as
\begin{align}\label{bl8j}
\begin{split}
\int\rho(B^2+L^2)\mbox{div}u\mbox{d}x&=\int\rho(B^2+L^2)\frac{F+\frac12|H|^2+p(\rho)}{2\mu+\lambda(\rho)}\mbox{d}x\\
&\leq CY(t)\psi(t)(1+Y(t))\\
&\leq \delta\psi^2(t)+C_{\delta}(Y(t)+1)^4.
\end{split}
\end{align}
Direct estimates give
\begin{align}\label{bl8jj}
\begin{split}
&-\int\rho(2\mu+\lambda(\rho))
[(F+\frac{1}{2}|H|^2)(\frac{1}{2\mu+\lambda(\rho)})'
+(\frac{p}{2\mu+\lambda(\rho)})']\mbox{div}u(B_{x_2}+L_{x_2})\mbox{d}x\\
&\leq  \delta\int(2\mu+\lambda(\rho))(B_{x_2}+L_{x_2})^2\mbox{d}x+C_{\delta}\int\rho^2(2\mu+\lambda(\rho))
[(F+\frac{1}{2}|H|^2)(\frac{1}{2\mu+\lambda(\rho)})'
\\
&+(\frac{p}{2\mu+\lambda(\rho)})']^2(\mbox{div}u)^2\mbox{d}x\\
&\leq \delta\psi^2(t)+C_{\delta}\int\rho^2[(F+\frac{1}{2}|H|^2)(\frac{1}{2\mu+\lambda(\rho)})'
+(\frac{p}{2\mu+\lambda(\rho)})']^2\frac{|F+\frac12|H|^2|^2+p^2(\rho)}{2\mu+\lambda(\rho)}\mbox{d}x\\
&\leq \delta\psi^2(t)+C_{\delta}(1+\|F+\frac12|H|^2\|_4^4)\\
&\leq \delta\psi^2(t)+C_{\delta}(Y(t)+1)^4,
 \end{split}
\end{align}
\begin{align}\label{bl9}
\begin{split}
&-\int B(u_{x_2}\cdot\nabla(F+\frac{1}{2}|H|^2)+\mu u_{x_1}\nabla\omega)+L(u_{x_1}\cdot\nabla(F+\frac{1}{2}|H|^2)-\mu u_{x_2}\nabla\omega)\mbox{d}x\\
&\leq C\int|(B,L)||\nabla u||\nabla(F+\frac12|H|^2,\omega)|\mbox{d}x
\leq C\|(B,L)\|_8\|\nabla u\|_2\|\nabla(F+\frac12|H|^2,\omega)\|_{\frac83}\\
&\leq C\|\nabla(B,L)\|_2\|\rho(B,L)\|_{\frac83}\leq CY^{\frac38}(t)\psi^{\frac{13}{8}}
\leq\delta\psi^2(t)+C_{\delta}Y^2(t),
\end{split}
\end{align}
\begin{align}\label{bl10}
\begin{split}
&+\int\mu[\omega\mbox{div}u+(\frac{1}{\rho}H\cdot\nabla H_1)_{x_2}-(\frac{1}{\rho}H\cdot\nabla H_1)_{x_1}](B_{x_1}-L_{x_2})\mbox{d}x\\
&\leq \mu\big(\int(B_{x_1}-L_{x_2})^2\mbox{d}x\big)^{\frac12}\big(\int\omega^2(\mbox{div}u)^2+[(\frac{1}{\rho}H\cdot\nabla H_1)_{x_2}-(\frac{1}{\rho}H\cdot\nabla H_1)_{x_1}]^2\mbox{d}x\big)^{\frac12}\\
&\leq \delta\psi^2(t)+C_{\delta}\|\omega\|_4^2\|\frac{F+\frac12|H|^2+p(\rho)}{2\mu+\lambda(\rho)}\|_4^2+C_{\delta}Y^2(t)\\
&\leq\delta\psi^2(t)+C_{\delta}\|\omega\|_4^2(1+\|F+\frac12|H|^2\|_4^2)+C_{\delta}Y^2(t)\\
&\leq\delta\psi^2(t)+C_{\delta}(Y(t)+1)^4,
\end{split}
\end{align}
\begin{align}\label{bl11}
\begin{split}
&\int(2\mu+\lambda(\rho))[(u_{1x_1})^2+2u_{1x_2}u_{2x_1}
+(u_{2x_2})^2](B_{x_2}+L_{x_1})\mbox{d}x\\&-\int\frac{2\mu+\lambda(\rho)}{\rho}
[(H_{1x_1})^2+2H_{1x_2}H_{2x_1}+(H_{2x_2})^2](B_{x_2}+L_{x_1})\mbox{d}x\\
&\leq \big(\int(2\mu+\lambda(\rho))(B_{x_2}+L_{x_1})^2\mbox{d}x\big)^{\frac12}\big(\int(2\mu+\lambda(\rho))[(u_{1x_1})^2+2u_{1x_2}u_{2x_1}
\\
&+(u_{2x_2})^2+(H_{1x_1})^2+2H_{1x_2}H_{2x_1}+(H_{2x_2})^2]^2\mbox{d}x\big)^{\frac12}\\
&\leq \delta\psi^2(t)+C_{\delta}\int(2\mu+\lambda(\rho))[(u_{1x_1})^2+2u_{1x_2}u_{2x_1}
+(u_{2x_2})^2+(H_{1x_1})^2\\
&+2H_{1x_2}H_{2x_1}+(H_{2x_2})^2]^2\mbox{d}x\\
&\leq \delta\psi^2(t)+C_{\delta}\|2\mu+\lambda(\rho)\|_2(\|\nabla u\|_8^4+\|\nabla H\|_8^4)\\
&\leq \delta\psi^2(t)+C_{\delta}(\|\mbox{div} u\|_8^4+\|\omega\|_8^4+\|\mbox{curl}H\|_8^4)\\
&\leq \delta\psi^2(t)+C_{\delta}(1+\|(F+\frac12|H|^2,\omega)\|_8^4+\|\mbox{curl}H\|_8^4)\\
&\leq \delta\psi^2(t)+C_{\delta}(1+\|\nabla(F+\frac12|H|^2,\omega,\mbox{curl}H)\|_{\frac85}^4)\\
&\leq\delta\psi^2(t)+C_{\delta}(Y(t)+1)^4,
\end{split}
\end{align}
where we have used the fact that
\begin{align*}
\|\nabla u\|_2\leq C(\|\mbox{div}u\|_2+\|\omega\|_2)\leq C(\|\frac{F+\frac12|H|^2+p(\rho)}{2\mu+\lambda(\rho)}\|_2+\|\omega\|_2)\leq C,
\end{align*}
\begin{align*}
\|\rho(B,L)\|_{\frac83}&=(\int\sqrt{\rho}|(B,L)||(B,L)|^{\frac53}\rho^{\frac{13}{6}}\mbox{d}x)^{\frac38}
\leq\|\sqrt{\rho}(B,L)\|_2^{\frac38}\|(B,L)\|_4^{\frac58}\|\rho\|_{26}^{\frac{13}{16}}\\
&\leq CY^{\frac38}(t)\|\nabla(B,L)\|_2^{\frac58}.
\end{align*}
After a tedious calculation,
\begin{align}\label{bl12}
\begin{split}
&-\int(2\mu+\lambda(\rho)) [(\frac{1}{\rho})_{x_1} H \cdot \nabla H_1+(\frac{1}{\rho})_{x_2}H\cdot \nabla H_2](B_{x_2}+L_{x_2})\mbox{d}x
\\&+\int(2\mu+\lambda(\rho))(B_{x_2}+L_{x_1})[(\frac{1}{2\rho}|H|^2_{x_2})_{x_2}
+(\frac{1}{2\rho}|H|^2_{x_1})_{x_1}]\\
&\int\mu[(\frac{1}{2\rho}|H|^2_{x_2})_{x_1}-(\frac{1}{2\rho}|H|^2_{x_1})_{x_2}](B_{x_1}-L_{x_2})\mbox{d}x\\
&\leq \delta\psi^2(t)+C_{\delta}(Y(t)+1)^4,
\end{split}
\end{align}
\begin{align}\label{bl13}
\begin{split}
&-\int\nabla\mbox{curl}u\cdot\nabla H\cdot\nabla\mbox{curl}H
+\mbox{curl}u\cdot\nabla \nabla H\cdot\nabla\mbox{curl}H+\nabla u\cdot\nabla\mbox{curl}H\cdot\nabla\mbox{curl}H\\&+u\cdot\nabla\nabla\mbox{curl}H\cdot\nabla\mbox{curl}H
-\nabla\mbox{curl}H\cdot\nabla u\cdot\nabla\mbox{curl}H-\mbox{curl}H\cdot\nabla\nabla u\cdot\nabla\mbox{curl}H\\
&-\nabla H\cdot \nabla\mbox{curl}u\cdot\nabla\mbox{curl}H-H\cdot\nabla\nabla\mbox{curl}u\cdot\nabla\mbox{curl}H
+\mbox{div}u\nabla\mbox{curl}H\cdot\nabla\mbox{curl}H\\&+\mbox{curl}H\cdot\nabla\mbox{div}u\cdot\nabla\mbox{curl}H
+\nabla H\mbox{curl}\mbox{div}u\cdot\nabla\mbox{curl}H+H\cdot\nabla\mbox{curl}\mbox{div}u\cdot\nabla\mbox{curl}H\mbox{d}x\\
&\leq \delta\psi^2(t)+C(1+Y(t))^4,
\end{split}
\end{align}
\begin{align}\label{bl14}
-\int\nabla u|\nabla\rho|^2\mbox{d}x-\frac12\int\mbox{div}u|\nabla\rho|^2\mbox{d}x
-\int\rho\nabla\rho\cdot\nabla\mbox{div}u\mbox{d}x\leq \delta\psi^2(t)+C(1+Y(t))^4.
\end{align}
Incorporating \eqref{bl8j}-\eqref{bl14} with \eqref{bl3} leads to
\begin{align*}
\frac12\frac{\mbox{d}}{\mbox{d}x}Y^2(t)+\psi^2(t)\leq8\delta\psi^2(t)+C_{\delta}(1+Y^2(t))^2.
\end{align*}
Choosing $8\delta=\frac12$, from Lemma \ref{lem3.5}, we get $Y^2(t)\in L^1(0,T)$, and so using Gronwall's inequality gives that
\begin{align}\label{bl20}
Y^2(t)+\int_0^T\psi^2(t)\mbox{d}t\leq Y^2(0)+C.
\end{align}
From the system \ref{eq_EP_1}, it holds that
\begin{align*}
\mathcal{L}_{\rho_0}u_0&=\mu\Delta u_0+\nabla((\mu+\lambda(\rho_0))\mbox{div}u_0)=\mu\Delta u_0+\nabla(F_0-\mu\mbox{div}u_0+p(\rho_0)+\frac12|H_0|)\\
&=[\mu\nabla(\mbox{div}u_0)-\mu\nabla\times(\nabla\times u_0)]+\nabla(F_0-\mu\mbox{div}u_0+p(\rho_0)+\frac12|H_0|)
\end{align*}
with
\begin{align*}
F_0=(2\mu+\lambda(\rho_0))\mbox{div}u_0-p(\rho_0)-\frac12|H_0|,
\end{align*}
similarly one can define $\omega_0$, $B_0$, $L_0$.
Thus
\begin{align*}
\mathcal{L}_{\rho_0}u_0-\nabla p_0&=\nabla(F_0+\frac12|H_0|)-\mu\nabla\times(\nabla\times u_0)\\
&=\nabla(F_0+\frac12|H_0|)-\mu(\partial_{x_2}\omega_0,-\partial_{x_1}\omega_0)^t\\
&=((F_0+\frac12|H_0|)_{x_1}-\mu\partial_{x_2}\omega_0,(F_0
+\frac12|H_0|)_{x_2}+\mu\partial_{x_1}\omega_0)\\
&=\rho_0(L_0,B_0)^t.
\end{align*}
Hence, there exists $g\in L^2({\mathbb{T}^2})$ such that
\begin{align*}
\sqrt{\rho_0}g=\rho_0(L_0,B_0)^t,
\end{align*}
which gives that
\begin{align*}
Y^2(0)=\|\sqrt{\rho_0}(L_0,B_0)\|_2^2=\|\frac{\sqrt{\rho_0}g}{\sqrt{\rho_0}}\|_2^2\leq C.
\end{align*}
Therefore, from \eqref{bl20}, it holds that
\begin{align*}
Y^2(t)+\int_0^T\psi^2(t)\mbox{d}t\leq C.
\end{align*}
\end{proof}
The Lemma \ref{lem3.6} is proved.

Step 5. Upper bound of the density
\begin{lem}\label{lem3.7}
It holds that
\begin{align*}
\int_0^T\|(F+\frac12|H|^2,\omega)\|^3_{L^{\infty}}\mbox{d}t\leq C.
\end{align*}
\end{lem}
\begin{proof}
From \eqref{bl7} with $p=3$, it holds that
\begin{align}
\begin{split}
\int_0^T\|\nabla(F+\frac12|H|^2,\omega)\|_3^3\mbox{d}t&\leq C\int_0^T\|\rho(B,L)\|_3^3\mbox{d}t\\
&\leq C\int_0^T\int\rho^3|(B,L)|^3\mbox{d}x\mbox{d}t
\\
&\leq C\int_0^T\int\rho^{\frac12}|(B,L)||(B,L)|^2\rho^{\frac52}\mbox{d}\mbox{d}t
\\&\leq C\int_0^T\|\rho^{\frac12}(B,L)\|_2\|(B,L)\|_8^2\|\rho\|_{10}^{\frac25}\mbox{d}t\\
&\leq C\int_0^T\|\nabla(B,L)\|_2^2\mbox{d}t\\
&\leq C\int_0^T\psi^2(t)\mbox{d}t
\leq C,
\end{split}
\end{align}
which together with Lemma \ref{lem3} gives that
\begin{align}\label{lem3.94}
\int_0^T\|(F+\frac12|H|^2,\omega)\|_{\infty}^3\leq\int_0^T\|(F+\frac12|H|^2,\omega)\|_{W^{1,3}}^3\mbox{d}t\leq C.
\end{align}
\end{proof}
\begin{lem}\label{lem3.8}
It holds that
\begin{align*}
\rho(x,t)\leq C, \ \forall (x,t)\in \mathbb{T}^2\times[0,T].
\end{align*}
\end{lem}
\begin{proof}
Using the continuity equation of \eqref{eq_EP_1}, and the definition of $\theta(\rho)$, one gets
\begin{align}
\theta(\rho)_t+u\cdot\nabla\rho+F+\frac12|H|^2+p(\rho)=0.
\end{align}
Introducing the particle path $\vec{X}(x,t;\tau)$ through the point $(x,t)\in\mathbb{T}\times[0,T]$ defined by
\begin{align}
\begin{cases}
\frac{\mbox{d}\vec{X}(x,t;\tau)}{\mbox{d}\tau}=u(\vec{X}(x,t;\tau),\tau),\\
\vec{X}(x,t;\tau)|_{\tau=t}=x,
\end{cases}
\end{align}
then we can show that
\begin{align}
\frac{\mbox{d}}{\mbox{d}\tau}\theta(\rho)(\vec{X}(x,t;\tau),\tau)
=-p(\rho)(\vec{X}(x,t;\tau),\tau)-F(\vec{X}(x,t;\tau),\tau)-\frac12|H|^2(\vec{X}(x,t;\tau),\tau).
\end{align}
Integrating the above equality over $[0,t]$ implies that
\begin{align}\label{density1}
\theta(\rho)(x,t)-\theta(\rho_0)(\vec{X}_0)=-\int_0^t(p(\rho)+F+\frac12|H|^2)(\vec{X}(x,t;\tau),\tau)\mbox{d}\tau,
\end{align}
with $\vec{X}_0=\vec{X}(x,t;\tau)|_{\tau=0}$.

From \eqref{density1}, it follows that
\begin{align*}
&2\mu\ln\frac{\rho(x,t)}{\rho_0(\vec{X}_0)}+\frac{1}{\beta}\rho^{\beta}(x,t)+\int_0^tp(\rho)(\vec{X}(x,t;\tau),\tau)\mbox{d}\tau\\
&=\frac{1}{\beta}\rho_0^{\beta}(\vec{X}_0)-\int_0^t(F+\frac12|H|^2)(\vec{X}(x,t;\tau),\tau)\mbox{d}\tau.
\end{align*}
Thus
\begin{align*}
&2\mu\ln\frac{\rho(x,t)}{\rho_0(\vec{X}_0)}\leq\frac{1}{\beta}\|\rho_0\|_{\infty}^{\beta}
+\int_0^t\|(F+\frac12|H|^2)(\tau,\cdot)\|_{\infty}\mbox{d}\tau\leq C,
\end{align*}
which yields that
\begin{align*}
\rho(x,t)\leq C\rho_0(\vec{X}_0).
\end{align*}
Hence,
\begin{align*}
\rho(x,t)\leq C, \ \forall(x,t)\in\mathbb{T}^2\times[0,T].
\end{align*}
\end{proof}
\begin{lem}\label{lem3.9}
It holds that for any $1<p<+\infty$,
\begin{align*}
\int_0^T\|\mbox{div}u\|_{\infty}^3+\|\nabla(F+\frac12|H|^2,\omega)\|_{p}^2\mbox{d}t\leq C, \\
\rho(x,t)\geq m_1>0,\ \mbox{with} \ m_1\ \mbox{a positive constant}.
\end{align*}
\end{lem}
\begin{proof}
For any $1<p<+\infty$,
\begin{align*}
\begin{split}
\int_0^T\|\nabla(F+\frac12|H|^2,\omega)\|_{p}^2\mbox{d}t&\leq C\int_0^T\|\rho(B,L)\|_p^2\mbox{d}t\leq C\int_0^T\|(B,L)\|_p^2\mbox{d}t\\
&\leq C\int_0^T\|\nabla(B,L)\|_2^2\mbox{d}t\leq C.
\end{split}
\end{align*}
On the other hand, it holds that
\begin{align*}
\int_0^T\|\mbox{div}u\|_{\infty}^3\mbox{d}t\leq C\int_0^T\|F+\frac12|H|^2\|_{\infty}^3+\|p(\rho)\|_{\infty}^3\mbox{d}t\leq C.
\end{align*}
From the continuity equation and the above estimate, it is quite easy to show that
$$\rho(x,t)\geq m_1>0.$$
Thus we complete the proof of Lemma \ref{lem3.9}.
\end{proof}

\section{Higher order estimates} \label{sec_g100}
In this section we derive some uniform estimates on their higher order estimates by virtue of the approximate solutions and basic estimates.
\begin{lem}\label{lem4.1}
It holds that for any $1\leq p<+\infty$,
\begin{align}
\sup_{t\in[0,T]}\|(\nabla\rho,\nabla p(\rho))(t,\cdot)\|_p+\int_0^T\|\nabla u\|_{\infty}^2\mbox{d}t\leq C.
\end{align}
\end{lem}
\begin{proof}
Applying the operator $\nabla$ to the continuity equation of \eqref{eq_EP_1}, then multiplying the resulted equation by $p|\nabla\rho|^{p-2}\nabla \rho$ with $p\geq 2$, integrating it in the space variable $x$ over $\mathbb{T}^2$ implies that
\begin{align*}
\frac{\mbox{d}}{\mbox{d}t}\|\nabla\rho\|_p^p&
=-(p-1)\int|\nabla\rho|^p\mbox{div}u\mbox{d}x-p\int|\nabla\rho|^{p-2}\nabla\rho\cdot\nabla u\cdot\nabla\rho\mbox{d}x\\&-p\int\rho|\nabla\rho|^{p-2}\nabla\rho\cdot\nabla\mbox{div}u\mbox{d}x\\
&\leq (p-1)\|\mbox{div}u\|_{\infty}\|\nabla\rho\|_p^p+p\|\nabla u\|_{\infty}\|\nabla\rho\|_p^p+p\|\nabla\mbox{div}u\|_{p}\|\nabla\rho\|_p^{p-1}\|\rho\|_{\infty}.
\end{align*}
 This gives that
 \begin{align}\label{density4.5}
 \begin{split}
 \frac{\mbox{d}}{\mbox{d}t}\|\nabla\rho\|_p &\leq C[\|\nabla u\|_{\infty}\|\nabla\rho\|_p+\|\nabla\mbox{div}u\|_{p}]\\&\leq C[\|\nabla u\|_{\infty}\|\nabla\rho\|_p+\|\nabla(\frac{F+\frac12|H|^2+p(\rho)}{2\mu+\lambda(\rho)})\|_{p}]\\
 &\leq C[(\|\nabla u\|_{\infty}+\|F+\frac12|H|^2\|_{\infty}+1)\|\nabla\rho\|_p+\|\nabla(F+\frac12|H|^2)\|_{p}].
 \end{split}
 \end{align}
 Since
 \begin{align}\label{elliptic4.6}
 \mathcal{L}_{\rho}u=\nabla p(\rho)+\rho(L,B)^t,
 \end{align}
 by elliptic estimates and \eqref{bl3.74}, we show that for any $1<p<+\infty$,
 \begin{align*}
 \|\nabla^2u\|_p&\leq C[\|\nabla p(\rho)\|_p+\|\rho(L,B)\|_p]\\&\leq C[\|\nabla \rho\|_p+\|(L,B)\|_p]\\&\leq C[\|\nabla \rho\|_p+\|\nabla(L,B)\|_2].
 \end{align*}
 From the Beal-Kato-Majda type inequality, it follows that
  \begin{align}\label{lem4.8}
  \begin{split}
 \|\nabla u\|_{\infty}&\leq C(\|\mbox{div}u\|_{\infty}+\|\omega\|_{\infty})\ln(\mbox{e}+\|\nabla^2u\|_3)\\&\leq C(\|\mbox{div}u\|_{\infty}+\|\omega\|_{\infty})\ln(\mbox{e}+\|\nabla\rho\|_3)
 \\
 &+C(\|\mbox{div}u\|_{\infty}+\|\omega\|_{\infty})\ln(\mbox{e}+\|\nabla(L,B)\|_2),
 \end{split}
 \end{align}
 which together with \eqref{density4.5} for $p=3$, leads to
 \begin{align*}
 \begin{split}
 \frac{\mbox{d}}{\mbox{d}t}\|\nabla\rho\|_3 &\leq C(\|\mbox{div}u\|_{\infty}+\|\omega\|_{\infty})\ln(\mbox{e}+\|\nabla\rho\|_3)\|\nabla\rho\|_3
 \\
 &+C[(\|\mbox{div}u\|_{\infty}+\|\omega\|_{\infty})\ln(\mbox{e}
 +\|\nabla(L,B)\|_2)\\
 &+\|F+\frac12|H|^2\|_{\infty}+1]\|\nabla\rho\|_3+C\|\nabla(F+\frac12|H|^2)\|_{p}.
 \end{split}
 \end{align*}
 From \eqref{lem3.94}, Lemma \ref{lem3.9} and the Gronwall's inequality, we get
 \begin{align}\label{nablarho}
 \sup_{t\in[0,T]}\|\nabla\rho\|_3\leq C.
 \end{align}
  Incorporating Lemma \ref{lem3.7}, Lemma \ref{lem3.9} and \eqref{lem4.8}-\eqref{nablarho} yields that
 \begin{align}
 \int_0^T\|\nabla u\|_{\infty}^2\mbox{d}t\leq C.
 \end{align}
 Again using Lemma \ref{lem3.7}, Lemma \ref{lem3.9} and \eqref{density4.5}, by Gronwall's inequality, one arrives at
 \begin{align}
 \sup_{t\in[0,T]}\|\nabla\rho\|_p\leq C(\|\nabla\rho_0\|_p+1),\ \forall p\in[1,+\infty).
 \end{align}
\end{proof}
\begin{lem}\label{lem4.2}
It holds that for any $1\leq p<+\infty$,
\begin{align*}
&\sup_{t\in[0,T]}[\|u(t,\cdot)\|_{\infty}+\|\nabla u\|_p+\|(\rho_t, p_t)\|_p+\|(\rho_t, p_t)\|_{H^1}+\|(\rho,p(\rho),u)\|_{H^2}]\\&+\int_0^T\|u\|_{H^3}^2\mbox{d}t\leq C.
\end{align*}
\end{lem}
\begin{proof}
From elliptic estimates and \eqref{elliptic4.6}, it holds that
\begin{align}
\begin{split}
\sup_{t\in[0,T]}\|u\|_{H^2}&\leq C\sup_{t\in[0,T]}(\|\nabla p(\rho)\|_2+\|\rho(L,B)\|_2)\\&\leq C\sup_{t\in[0,T]}(\|\nabla p(\rho)\|_2+\|\sqrt{\rho}(L,B)\|_2)\leq C.
\end{split}
\end{align}
By virtue of Sobolev embedding theorem, one gets
\begin{align}\label{elliptic4.15}
\sup_{t\in[0,T]}|u(x,t)|\leq C,\ \sup_{t\in[0,T]}\|\nabla u\|_p\leq C,\ \forall1\leq p<+\infty.
\end{align}
From
\begin{align*}
\rho_t=-u\cdot\nabla\rho-\rho\mbox{div}u
\end{align*}
and
\begin{align}\label{pressure}
p_t=-u\cdot\nabla p-\rho p'(\rho)\mbox{div}u,
\end{align}
together with the uniform upper bound of the density, Lemma \ref{lem4.1} and \eqref{elliptic4.15}, one can show that
\begin{align}
\sup_{t\in[0,T]}\|(\rho_t,p_t)\|_p\leq C, \ \forall p\in [1,+\infty).
\end{align}
Applying the operator $\nabla^2$ to the continuity equation in \eqref{eq_EP_1} and \eqref{pressure}, multiplying the resulted equation by $\nabla^2\rho$ and $\nabla^2p(\rho)$, integrating them over the torus $\mathbb{T}^2$, we can prove that
\begin{align}\label{density}
\begin{split}
\frac{\mbox{d}}{\mbox{d}t}\|\nabla^2\rho\|_2^2&\leq C[\|\nabla u\|_{\infty}\|\nabla^2\rho\|_2^2+\|\nabla\rho\|_3\|\nabla^2\rho\|_2\|\nabla^2u\|_6
+\|\rho\|_{L^{\infty}}\|\nabla^2\rho\|_2\|\nabla^3u\|_2]\\
&\leq C[(\|\nabla u\|_{\infty}+1)\|\nabla^2\rho\|_2^2+\|\nabla^3u\|_2^2+1],
\end{split}
\end{align}
\begin{align}\label{pressure1}
\begin{split}
\frac{\mbox{d}}{\mbox{d}t}\|\nabla^2p(\rho)\|_2^2&\leq C[\|\nabla^2u\|_6\|\nabla p\|_3\|\nabla^2p\|_2+\|\nabla\rho\|_3\|\nabla^2p\|_2\|\nabla^2u\|_6\|p'(\rho)\|_{\infty}
\\&+\|\nabla u\|_{\infty}\|\nabla^2p\|_2^2+\|p''(\rho)\rho\|_{L^{\infty}}\|\nabla^2u\|_6\|\nabla^2p\|_2\|\nabla\rho\|_6^2\\
&+\|p'(\rho)\rho\|_{L^{\infty}}\|\nabla^3u\|_2\|\nabla^2p\|_2+\|p'(\rho)\|_{L^{\infty}}\|\nabla u\|_{\infty}\|\nabla^2p\|_2\|\nabla^2\rho\|_2\\&+\|p''(\rho)\|_{L^{\infty}}\|\nabla u\|_{\infty}\|\nabla^2p\|_2\|\nabla\rho\|_4^2+\|p'''(\rho)\rho\|_{L^{\infty}}\|\nabla u\|_{\infty}\|\nabla^2p\|_2\|\nabla\rho\|_4^2\\
&+\|p''(\rho)\rho\|_{L^{\infty}}\|\nabla u\|_{\infty}\|\nabla^2p\|_2\|\nabla^2\rho\|_2]\\
&\leq C[(\|\nabla u\|_{\infty}+1)\|\nabla^2p\|_2^2+\|\nabla^3u\|_2^2+\|\nabla u\|_{\infty}\|\nabla^2\rho\|_2^2].
\end{split}
\end{align}
From \eqref{elliptic4.6}, it holds that
\begin{align*}
\mathcal{L}(\nabla u)=\nabla^2p(\rho)+\nabla[\rho(L,B)]+\nabla (\nabla\lambda(\rho)\mbox{div}u):=\phi.
\end{align*}
Then the standard elliptic estimates imply that
\begin{align*}
\|u\|_{H^3}&\leq C[\|u\|_{H^1}+\|\phi\|_{2}]\\
&\leq C[\|u\|_{H^1}+\|\nabla^2p\|_{2}+\|\nabla\rho\|_{4}\|(L,B)\|_{4}+\|\rho\|_{\infty}\|\nabla(L,B)\|_{2}\\
&+\|\nabla^2\rho\|_2\|\nabla u\|_{\infty}+\|\nabla\rho\|_3\|\nabla^2u\|_6],
\end{align*}
and
\begin{align*}
\|\nabla^2u\|_6\leq C(\|\nabla p{\rho}\|_{6}+\|\rho(L,B)\|_{6})\leq C(1+\|\nabla(L,B)\|_{2}).
\end{align*}
Hence,
\begin{align*}
\|u\|_{H^3}\leq C[1+\|\nabla^p\|_{2}+\|\nabla(L,B)\|_{2}+\|\nabla u\|_{\infty}\|\nabla^2\rho\|_2],
\end{align*}
which together with \eqref{density}-\eqref{pressure1} gives that
\begin{align*}
\frac{\mbox{d}}{\mbox{d}t}\|(\nabla^2\rho,\nabla^2p(\rho))\|_2^2&\leq C[(\|\nabla u\|_{\infty}^2+1)\|(\nabla^2\rho,\nabla^2p(\rho))\|_2^2+\|\nabla(L,B)\|_{2}+1].
\end{align*}
Using the Gronwall's inequality, one can show that
\begin{align*}
\|(\nabla^2\rho,\nabla^2p(\rho))\|_2^2&\leq (\|(\nabla^2\rho_0,\nabla^2p_0)\|_2^2+C\int_0^T(\|\nabla(L,B)\|_{2}+1)\mbox{d}t)\mbox{e}^{\int_0^T(\|\nabla u\|_{\infty}^2+1)\mbox{d}t}\\
&\leq C,
\end{align*}
which gives that
\begin{align*}
\sup_{t\in[0,T]}(\|(\rho, p(\rho))\|_{H^2}+\|(\rho_t, p_t)\|_{H^1})+\int_0^T\|u\|_{H^3}^2\mbox{d}t\leq C.
\end{align*}
The proof of Lemma \ref{lem4.2} is completed.
\end{proof}
\begin{lem}\label{lem4.3}
It holds that
\begin{align*}
\int_0^T\|H_t\|_2^2\mbox{d}t\leq C,
\end{align*}
\begin{align*}
\sup_{t\in[0,T]}\|\sqrt{\rho}u_t\|_2^2\mbox{d}t+\int_0^T\|u_t\|_{H^1}^2\mbox{d}t\leq C.
\end{align*}
\end{lem}
\begin{proof}
From the magnetic equation in \eqref{BL}, one gets
\begin{align*}
\begin{split}
\|H_t\|_2&\leq C[\|u\cdot\nabla H\|_2+\|\Delta H\|_2+\|H\cdot\nabla u\|_2+\|H\mbox{div}u\|_2]\\
&\leq C[\|u\|_{\infty}\|\nabla H\|_2+\|\Delta H\|_2+\|\nabla u\|_2\|H\|_{\infty}],
\end{split}
\end{align*}
which yields that
\begin{align*}
\int_0^T\|H_t\|_{2}^2\mbox{d}t\leq C.
\end{align*}
The momentum equation in \eqref{eq_EP_1} can be rewritten as
\begin{align}
\rho u_t+\rho u\cdot\nabla u+\nabla p(\rho)+\nabla (\frac12|H|^2)-H\cdot \nabla H=\mathcal{L}_{\rho}u:=\mu\Delta u+\nabla((\mu+\lambda(\rho))\mbox{div}u).
\end{align}
Applying $\partial_t$ to the above equation gives that
\begin{align}\label{lem4.26}
\begin{split}
&\rho u_{tt}+\rho u\cdot\nabla u_t+\nabla p(\rho)_t+\nabla (\frac12|H|^2_t)-H_t\cdot \nabla H-H\cdot\nabla H_t+\rho_tu_t\\
&=\mu\Delta u_t+\nabla((\mu+\lambda(\rho))\mbox{div}u_t)-\rho_tu\cdot\nabla u-\rho u_t\cdot\nabla u+\nabla(\lambda(\rho)_t\mbox{div}u).
\end{split}
\end{align}
Multiplying the above equation by $u_t$ and integrating the resulting equation over $\mathbb{T}^2$ gives that
\begin{align}\label{higerorder}
\begin{split}
&\frac12\frac{\mbox{d}}{\mbox{d}t}\int\rho|u_t|^2\mbox{d}x+\int(\mu|\nabla u_t|^2+(\mu+\lambda(\rho))|\mbox{div} u_t|^2)\mbox{d}x\\
&=-\int\nabla p(\rho)_t\cdot u_t\mbox{d}x-\int\nabla(\frac12|H|^2_t)\cdot u_t\mbox{d}x+\int H_t\cdot \nabla H\cdot u_t\mbox{d}x\\
&+\int H\cdot \nabla H_t\cdot u_t\mbox{d}x-\int\rho_t|u_t|^2\mbox{d}x-\int\rho_tu\cdot\nabla u\cdot u_t\mbox{d}x\\
&-\int\rho(u_t\cdot\nabla u)\cdot u_t\mbox{d}x+\int\nabla(\lambda(\rho)_t\mbox{div}u)\cdot u_t\mbox{d}x.
\end{split}
\end{align}
The terms on the right-hand side of \eqref{higerorder} can be estimated as
\begin{align*}
-\int\nabla p(\rho)_t\cdot u_t\mbox{d}x=\int p(\rho)_t \mbox{div}u_t\mbox{d}x\leq \frac{\mu}{8}\|\mbox{div}u_t\|_2^2+C\| p(\rho)_t\|_2^2\leq\frac{\mu}{8}\|\mbox{div}u_t\|_2^2+C,
\end{align*}
\begin{align*}
-\int\nabla(\frac12|H|^2_t)\cdot u_t\mbox{d}x=\int \frac12|H|^2_t \mbox{div}u_t\mbox{d}x=\int H\cdot H_t\mbox{div}u_t\mbox{d}x\leq \frac{\mu}{8}\|\mbox{div}u_t\|_2^2+C\| H_t\|_2^2,
\end{align*}
\begin{align*}
\int H_t\cdot \nabla H\cdot u_t\mbox{d}x=-\int H_t\cdot H\cdot \nabla u_t\mbox{d}x\leq \frac{\mu}{8}\|\nabla u_t\|_2^2+C\| H_t\|_2^2,
\end{align*}
\begin{align*}
\int H\cdot \nabla H_t\cdot u_t\mbox{d}x=-\int H\cdot H_t\cdot \nabla u_t\mbox{d}x\leq \frac{\mu}{8}\|\nabla u_t\|_2^2+C\| H_t\|_2^2,
\end{align*}
\begin{align*}
&-\int\rho_t|u_t|^2\mbox{d}x=\int\mbox{div}(\rho u) |u_t|^2\mbox{d}x=-2\int\rho u\cdot \nabla u_t\cdot u_t\mbox{d}x\\
&\leq \|\nabla u_t\|_2\|\sqrt{\rho} u_t\|_2\|\sqrt{\rho}\|_{\infty}\|u\|_{\infty}\leq \frac{\mu}{8}\|\nabla u_t\|_2^2+C\|\sqrt{\rho} u_t\|_2\|\sqrt{\rho}\|_{\infty}^2\|u\|_{\infty}^2\\
&\leq \frac{\mu}{8}\|\nabla u_t\|_2^2+C\|\sqrt{\rho} u_t\|_2,
\end{align*}
\begin{align*}
&-\int\rho_tu\cdot\nabla u\cdot u_t\mbox{d}x=\int\mbox{div}(\rho u)[(u\cdot \nabla u)\cdot u_t]\mbox{d}x=-\int\rho u\cdot\nabla[(u\cdot \nabla u)\cdot u_t]\mbox{d}x\\
&\leq \|\nabla u_t\|_2\|\nabla u\|_2\|\rho\|_{\infty}\|u\|_{\infty}^2+(\|\nabla u\|_4^2+\| u\|_{\infty}\|\nabla^2 u\|_{2})\|\sqrt{\rho}u_t\|_2\|\sqrt{\rho}\|_{\infty}\|u\|_{\infty}\\&\leq \frac{\mu}{8}\|\nabla u_t\|_2^2+C(\|\sqrt{\rho} u_t\|_2^2+\|\nabla u\|_{4}^4+\|(\nabla u, \nabla^2u)\|_{2}^2)\\
&\leq \frac{\mu}{8}\|\nabla u_t\|_2^2+C(\|\sqrt{\rho} u_t\|_2^2+1),
\end{align*}
\begin{align*}
-\int\rho(u_t\cdot\nabla u)\cdot u_t\mbox{d}x\leq \|\nabla u\|_{\infty}\|\sqrt{\rho} u_t\|_2^2,
\end{align*}
\begin{align*}
&\int\nabla(\lambda(\rho)_t\mbox{div}u)\cdot u_t\mbox{d}x=-\int\lambda(\rho)_t\mbox{div}u \mbox{div}u_t\mbox{d}x\\
&\leq \frac{\mu}{8}\|\mbox{div} u_t\|_2^2+C\|\lambda'{\rho}\|_{\infty}^2\|\rho_t\|_{H^1}^2\|\mbox{div}u\|_4^2\leq \frac{\mu}{8}\|\mbox{div} u_t\|_2^2+C,
\end{align*}
which combining with \eqref{higerorder} implies that
\begin{align*}
\|\sqrt{\rho}u_t\|_2^2+\int_0^t\|\nabla u_t\|_2^2\mbox{d}\tau
&\leq \|\sqrt{\rho_0}u_t(0)\|_2^2+\frac{\mu}{8}\|\mbox{div} u_t\|_2^2+C\int_0^t(\|\nabla u\|_{\infty}+1)\|\sqrt{\rho}u_t\|_2^2\mbox{d}\tau\\
&+C\int_0^t\|H_t\|_2^2\mbox{d}\tau+C.
\end{align*}
From the momentum equation of (\ref{eq_EP_1}), we have
\begin{align*}
\|\sqrt{\rho_0}u_t(0)\|_2^2&\leq\|\frac{\sqrt{\rho_0}}{\sqrt{\rho_0}}g\|_2^2
+\|\sqrt{\rho_0}\|^2_{\infty}\|u_0\|^2_{\infty}\|\nabla u_0\|_2^2+\|\frac{1}{\sqrt{\rho_0}}\|^2_{\infty}\|H_0\|^2_{\infty}\|\nabla H_0\|^2_{2}\\
&+\| H_0\|^2_{\infty}\|\nabla H_0\|_2^2
\leq C.
\end{align*}
Therefore, by the Gronwall's inequality, we get
\begin{align*}
\sup_{t\in[0,T]}\|\sqrt{\rho}u_t\|_2^2\mbox{d}t+\int_0^T\|\nabla u_t\|_{2}^2\mbox{d}t\leq C.
\end{align*}
Note that
\begin{align*}
u_t=(L,B)^t-u\cdot\nabla u+\frac{1}{\rho}H\cdot\nabla H,
\end{align*}
then for any $1\leq p<+\infty$,
\begin{align*}
\int_0^T\|u_t\|_p^2\mbox{d}t&\leq\int_0^T\|(L,B)\|_p^2+\|u\|_{\infty}^2\|\nabla u\|_p^2+\|\frac{1}{\rho}\|_{\infty}^2\|H\|_{\infty}^2\|\nabla H\|_p^2\mbox{d}t\\
&\leq\int_0^T\|\nabla(L,B)\|_2^2+\|u\|_{\infty}^2\|\nabla u\|_p^2+\|\frac{1}{\rho}\|_{\infty}^2\|H\|_{\infty}^2\|\nabla H\|_p^2\mbox{d}t\leq C.
\end{align*}
Thus, we arrive at
\begin{align*}
\int_0^T\|u_t\|_{H^1}^2\mbox{d}t\leq C.
\end{align*}
Hence the proof of Lemma \ref{lem4.3} is finished.
\end{proof}
\begin{lem}\label{lem4.4}
It holds that
\begin{align*}
\sup_{t\in[0,T]}\|(\rho_t,p(\rho)_t,\lambda(\rho)_t)\|_{H^1}+\int_0^T\|(\rho_{tt},p(\rho )_{tt}, \lambda(\rho)_{tt})\|_2^2\mbox{d}t\leq C.
\end{align*}
\end{lem}
\begin{proof}
Based on the continuity equation, it holds that $$\rho_t=-u\cdot\nabla\rho-\rho\mbox{div}u$$ and $$\rho_{tt}=-u_t\cdot\nabla\rho-u\cdot\nabla\rho_t-\rho_t\mbox{div}u-\rho\mbox{div}u_t.$$ Clearly
\begin{align}
\sup_{t\in[0,T]}\|\nabla\rho_t\|_2\leq \sup_{t\in[0,T]}[\|\nabla\rho \|_4\|\nabla u\|_4+\|u\|_{\infty}\|\nabla^2\rho\|_2+\|\rho\|_{\infty}\|\nabla^2u\|_2]\leq C,
\end{align}
\begin{align}
\begin{split}
\int_0^T\|\rho_{tt}\|_2^2\mbox{d}t&\leq \int_0^T[\|u_t \|_4^2\|\nabla \rho\|_4^2+\|u\|_{\infty}^2\|\nabla\rho_t\|_2^2+\|\rho_t\|_{4}^2\|\nabla u\|_4^2+\|\rho\|_{\infty}^2\|\nabla u_t\|_2^2]\\
&\leq C\int_0^T(\|u_t\|_{H^1}^2+1)\mbox{d}t\leq C.
\end{split}
\end{align}
By a similar calculation, we also obtain
\begin{align*}
\sup_{t\in[0,T]}\|\nabla(p(\rho)_t,\lambda(\rho)_t)\|_{2}+\int_0^T\|(p(\rho )_{tt}, \lambda(\rho)_{tt})\|_2^2\mbox{d}t\leq C.
\end{align*}
Thus we finish the proof of Lemma \ref{lem4.4}.
\end{proof}
\begin{lem}\label{lem4.51}
It holds that for any $0<t\leq T$,
\begin{align*}
\sup_{t\in[0,T]}\|H\|_{H^2}^2+\int_0^T\|H\|_{H^3}^2\mbox{d}t\leq C.
\end{align*}
\end{lem}
\begin{proof}
Applying $\partial_{x_ix_j},\  i,\  j=1,\ 2$, to the magnetic equation in \eqref{BL}, multiplying it by $\partial_{x_ix_j}H$, and then integrating the resulted equation in the space variable $x$, we obtain
\begin{align*}
&\frac12\frac{\mbox{d}}{\mbox{d}t}\|\nabla^2H\|_2^2+\nu\|\nabla^3H\|_2^2=-\int[\partial_{x_ix_j}u\cdot\nabla H+\partial_{x_i}u\cdot\nabla H_{x_j}+u\cdot\nabla H_{x_ix_j}\\
&+\partial_{x_j}u\cdot\nabla H_{x_i}-\partial_{x_ix_j}H\cdot\nabla u-\partial_{x_i}H\cdot\nabla u_{x_j}-\partial_{x_j}H\cdot\nabla u_{x_i}-H\cdot\nabla u_{x_ix_j}\\
&+\partial_{x_ix_j}H\mbox{div}u+\partial_{x_i}H\mbox{div}u_{x_j}
+\partial_{x_j}H\mbox{div}u_{x_i}+H\mbox{div}u_{x_ix_j}]\cdot \partial_{x_ix_j}H\mbox{d}x\\
&\leq \frac{\nu}{2}\|\nabla^3H\|_2^2+C(\|\nabla u\|_{\infty}+1)\|\nabla^2H\|_2^2+C\|\nabla^3u\|_2^2+C,
\end{align*}
that is,
\begin{align*}
&\frac12\frac{\mbox{d}}{\mbox{d}t}\|\nabla^2H\|_2^2+\nu\|\nabla^3H\|_2^2\leq C(\|\nabla u\|_{\infty}+1)\|\nabla^2H\|_2^2+C\|\nabla^3u\|_2^2+C.
\end{align*}
It follows from the Gronwall's inequality that
\begin{align*}
\sup_{t\in[0,T]}\|\nabla^2H\|_2^2+\int_0^T\|\nabla^3H\|_2^2\mbox{d}t&\leq (\|\nabla^2H_0^{\delta}\|_2^2+C\int_0^T(\|\nabla^3u\|_2^2+1)\mbox{d}t)\\
&\times\mbox{e}^{\int_0^T(\|\nabla u\|_{\infty}+1)\mbox{d}t}\leq C.
\end{align*}

Therefore, together with Lemma \ref{lem3.5}, Lemma \ref{lem4.1}, Lemma \ref{lem4.2} and elementary estimate \eqref{elementary estimate}, it holds that
\begin{align*}
\sup_{t\in[0,T]}\|\nabla^2H\|_2^2+\int_0^T\|H\|_{H^3}^2\mbox{d}t\leq C.
\end{align*}
\end{proof}

\begin{lem}\label{lem4.5}
It holds that
\begin{align*}
&\sup_{t\in[0,T]}\|H_t\|_2^2+\int_0^T\|H_t\|_{H^1}\mbox{d}t\leq C,\\
&\sup_{t\in[0,T]}\|H\|_{H^2}^2+\int_0^T\|\nabla^2H\|_{q}^2\mbox{d}t\leq C,
\end{align*}
\begin{align}\label{b2}
\begin{split}
&\sup_{t\in[0,T]}[t\|u_t\|_{H^1}^2+t\|u\|_{H^3}^2+t\|(\rho_{tt}, p(\rho)_{tt}, \lambda(\rho)_{tt})\|_{2}^2+\|(\rho,p(\rho))\|_{W^{2,q}}]\\
&+\int_0^Tt[\|\sqrt{\rho}u_{tt}\|_{2}^2+\|u_t\|_{H^2}^2+\|u\|_{H^4}^2]\mbox{d}t\leq C.
\end{split}
\end{align}
\end{lem}
\begin{proof}
From the magnetic equation in \eqref{BL}, we see that
\begin{align*}
H_{tt}+u_t\cdot\nabla H+u\cdot\nabla H_t-\nu\Delta H_t-H_t\cdot\nabla u-H\cdot \nabla u_t+H_t\mbox{div}u+H\mbox{div}u_t=0.
\end{align*}
Multiplying the above equality by $H_t$ implies that
\begin{align}\label{b1}
\begin{split}
\frac12\frac{\mbox{d}}{\mbox{d}t}\|H_t\|_2^2+\nu\|\nabla H_t\|_2^2&=-\int u_t\cdot\nabla H\cdot H_t\mbox{d}x-\int u\cdot\nabla H_t\cdot H_t\mbox{d}x\\
&+\int H_t\cdot\nabla u\cdot H_t\mbox{d}x+\int H\cdot \nabla u_t\cdot H_t\mbox{d}x\\
&-\int |H_t|^2\mbox{div}u\mbox{d}x-\int H\mbox{div}u_t\cdot H_t\mbox{d}x.
\end{split}
\end{align}
The estimates of the terms on the right-hand side of \eqref{b1} are as follows:
\begin{align*}
-\int u_t\cdot\nabla H\cdot H_t\mbox{d}x&=\int \nabla H\cdot H_t \mbox{div}u_t\mbox{d}x+\int u_t\cdot H\cdot \nabla H_t\mbox{d}x\\
&\leq C(\|\nabla u_t\|_2^2+\|u_t\|_2^2+\|H_t\|_2^2)+\frac{\nu}{2}\|\nabla H_t\|_2^2\\
&\leq C(\| u_t\|_{H^1}^2+\|H_t\|_2^2)+\frac{\nu}{2}\|\nabla H_t\|_2^2,
\end{align*}
\begin{align*}
-\int u\cdot\nabla H_t\cdot H_t\mbox{d}x&=\frac12\int |H_t|^2 \mbox{div}u\mbox{d}x\leq C\|\nabla u\|_{\infty}\|H_t\|_2^2,
\end{align*}
\begin{align*}
\int H_t\cdot\nabla u\cdot H_t\mbox{d}x\leq C\|\nabla u\|_{\infty}\|H_t\|_2^2,
\end{align*}
\begin{align*}
\int H\cdot \nabla u_t\cdot H_t\mbox{d}x\leq C\|H\|_{\infty}\|\nabla u_t\|_2\|H_t\|_2\leq C(\|\nabla u_t\|_2^2+\|H_t\|_2^2),
\end{align*}
\begin{align*}
-\int |H_t|^2\mbox{div}u\mbox{d}x\leq C\|\nabla u\|_{\infty}\|H_t\|_2^2,
\end{align*}
\begin{align*}
-\int H\mbox{div}u_t\cdot H_t\mbox{d}x\leq C(\|\nabla u_t\|_2^2+\|H_t\|_2^2).
\end{align*}
Thus from \eqref{b1}, it is not difficult to prove that
\begin{align*}
\begin{split}
\frac{\mbox{d}}{\mbox{d}t}\|H_t\|_2^2+\|\nabla H_t\|_2^2\lesssim \|u_t\|_{H^1}^2+\|H_t\|_2^2(1+\|\nabla u\|_{\infty}),
\end{split}
\end{align*}
which by the Gronwall's inequality gives that
\begin{align*}
\sup_{t\in[0,T]}\|H_t\|_2^2+\int_0^T\|\nabla H_t\|_{2}^2\mbox{d}t\lesssim\mbox{e}^{\int_0^T(1+\|\nabla u\|_{\infty})\mbox{d}t}(\|H_t^{\delta}(0)\|_2^2+\int_0^T\|H_t\|_{H^1}^2\mbox{d}t)\leq C.
\end{align*}
By the magnetic equation of \eqref{BL}, we have
\begin{align*}
\|\nabla^2 H\|_2&\leq \|H_t\|_2+\|u\cdot\nabla H\|_2+\|H\cdot\nabla u\|_2+\|H\mbox{div}u\|_2\\
&\leq \|H_t\|_2+\| u\|_{\infty}\|\nabla H\|_2+\|H\|_{\infty}\|\nabla u\|_2\leq C.
\end{align*}
The estimate of $\int_0^T\|\nabla^2 H\|_q^2\mbox{d}t\leq C$ is fairly easy so we skip it.

It remains for us to show the estimate \eqref{b2}. For this,
Multiplying the equation \eqref{lem4.26} by $u_{tt}$, and then integrating in the space variable $x$ give that
\begin{align}\label{b3}
\begin{split}
&\|\sqrt{\rho}u_{tt}\|_2^2+\frac12\frac{\mbox{d}}{\mbox{d}t}\int\mu|\nabla u_t|^2+(\mu+\lambda(\rho))|\mbox{div}u_t|^2\mbox{d}x\\
&=\frac12\int\lambda(\rho)_t|\mbox{div}u_t|^2\mbox{d}x-\int[\nabla p+\frac12\nabla|H|^2_t-H_t\cdot\nabla H-H\cdot\nabla H_t\\
&+\rho_t u_t+\rho_t u\cdot\nabla u_t+\rho u\cdot \nabla u_t+\rho u_t\cdot\nabla u-\nabla(\lambda(\rho)_t\mbox{div}u)]\cdot u_{tt}\mbox{d}x.
\end{split}
\end{align}
Note that
\begin{align*}
&\int\nabla(\lambda(\rho)_t\mbox{div}u)\cdot u_{tt}\mbox{d}x=-\int\lambda(\rho)_t\mbox{div}u\mbox{div}u_{tt}\mbox{d}x\\
&=-\frac{\mbox{d}}{\mbox{d}t}\int\lambda(\rho)_t\mbox{div}u\mbox{div}u_{t}\mbox{d}x
+\int\lambda(\rho)_t|\mbox{div}u_{t}|^2+\lambda(\rho)_{tt}\mbox{div}u\mbox{div}u_{t}\mbox{d}x.
\end{align*}
Hence we write \eqref{b3} as
\begin{align}\label{b4}
\begin{split}
&\|\sqrt{\rho}u_{tt}\|_2^2+\frac12\frac{\mbox{d}}{\mbox{d}t}\int\mu|\nabla u_t|^2+(\mu+\lambda(\rho))|\mbox{div}u_t|^2+\lambda(\rho)_t\mbox{div}u\mbox{div}u_{t}\mbox{d}x\\
&=\frac32\int\lambda(\rho)_t|\mbox{div}u_t|^2\mbox{d}x-\int[\nabla p+\frac12\nabla|H|^2_t-H_t\cdot\nabla H-H\cdot\nabla H_t\\
&+\rho_t u_t+\rho_t u\cdot\nabla u_t+\rho u\cdot \nabla u_t+\rho u_t\cdot\nabla u]\cdot u_{tt}\mbox{d}x+\int\lambda(\rho)_{tt}\mbox{div}u\mbox{div}u_{t}\mbox{d}x.
\end{split}
\end{align}
Observe that $\lambda(\rho)$ satisfies the transport equation $\lambda(\rho)_t=-u\cdot\nabla \lambda(\rho)-\rho \lambda'(\rho)\mbox{div}u$, then it holds that
\begin{align}\label{lem4.42}
\begin{split}
&\frac32\int\lambda(\rho)_t|\mbox{div}u_t|^2\mbox{d}x=-\frac{3}{2}\int u\cdot\nabla \lambda(\rho)|\mbox{div}u_t|^2\mbox{d}x-\frac{3}{2}\int \rho \lambda'(\rho)|\mbox{div}u_t|^2\mbox{d}x\\
&=3\int  \lambda(\rho)\mbox{div}u_tu\cdot\nabla\mbox{div}u_t\mbox{d}x+\frac{3}{2}\int (\lambda(\rho)-\rho \lambda'(\rho))\mbox{div}u|\mbox{div}u_t|^2\mbox{d}x\\
&\leq C\|\lambda(\rho)u\|_{\infty}\|\mbox{div}u_t\|_2\|\nabla\mbox{div}u_t\|_2+C\|\lambda(\rho)-\rho \lambda'(\rho)\|_{\infty}\|\nabla u\|_{\infty}\|\lambda(\rho)u\|_{\infty}\|\mbox{div}u_t\|_2^2\\
&\leq C\|\mbox{div}u_t\|_2\|\nabla\mbox{div}u_t\|_2+C\|\nabla u\|_{\infty}\|\mbox{div}u_t\|_2^2.
\end{split}
\end{align}
From \eqref{lem4.26}, it follows that
\begin{align*}
\mathcal{L}_{\rho}u_t=\rho u_{tt}+\rho_tu_t+(\rho u\cdot\nabla u)_t+\nabla p(\rho)_t+\nabla (\frac12|H|^2_t)-(H\cdot\nabla H)_t-\nabla(\lambda(\rho)_t\mbox{div}u).
\end{align*}
Then the standard elliptic estimates show that
\begin{align}\label{lem4.43}
\begin{split}
\|\nabla^2u_t\|_2&\leq C[\|\sqrt{\rho}\|_{\infty}\|\sqrt{\rho}u_{tt}\|_2+\|\rho_t\|_4\|u_t\|_4+\|\rho_t\|_4\|\nabla u\|_4\|u\|_{\infty}\\
&+\|\nabla u\|_4\| u_t\|_4\|\rho\|_{\infty}+\|u\|_{\infty}\|\rho\|_{\infty}\|\nabla u_t\|_2+\|\nabla p(\rho)_t\|_2\\
&+\|\nabla H\|_4\|H_t\|_4+\|H\|_{\infty}\|\nabla H_t\|_2+\|\nabla \lambda(\rho)_t\|_2\|\nabla u\|_{\infty}\\
&+\| \lambda(\rho)_t\|_4\|\nabla^2 u\|_4]\\
&\leq C[\|\sqrt{\rho}u_{tt}\|_2+\| u_t\|_4+\|H_t\|_4+\|\nabla u_t\|_2+\|\nabla H_t\|_2\\
&+\|\nabla^2 u\|_4+\|\nabla u\|_{\infty}+1]\\
&\leq C[\|\sqrt{\rho}u_{tt}\|_2+\| u_t\|_4+\|H_t\|_4+\|\nabla u_t\|_2+\|\nabla H_t\|_2\\
&+\|\nabla^3 u\|_2+\|\nabla u\|_{\infty}+1].
\end{split}
\end{align}
Plugging \eqref{lem4.43} into \eqref{lem4.42} leads to
\begin{align}\label{lem4.44}
\begin{split}
&\frac32\int\lambda(\rho)_t|\mbox{div}u_t|^2\mbox{d}x\leq \frac{1}{16}\|\sqrt{\rho}u_{tt}\|_2^2+C(\|u_t\|_4^2+\|H_t\|_4^2+\|\nabla H_t\|_2^2\\&+\|\nabla^3 u\|_2^2+\|\nabla u\|_{\infty}^2)+C(\|\nabla u\|_{\infty}+1)\|\nabla u_t\|_2^2.
\end{split}
\end{align}
On the other hand, it holds that
\begin{align}
\begin{split}
&-\int\nabla p(\rho)_t\cdot u_{tt}\mbox{d}x=\int p(\rho)_t\mbox{div}u_{tt}\mbox{d}x\\
&=\frac{\mbox{d}}{\mbox{d}t}\int p(\rho)_t\mbox{div}u_t\mbox{d}x-\int p(\rho)_{tt}\mbox{div}u_t\mbox{d}x\\
&\leq \frac{\mbox{d}}{\mbox{d}t}\int p(\rho)_t\mbox{div}u_t\mbox{d}x+\|p(\rho)_{tt}\|_2^2+\|\mbox{div}u_t\|_2^2,
\end{split}
\end{align}
\begin{align}
\begin{split}
\int\rho_tu_t\cdot u_{tt}\mbox{d}x&=-\int\rho_t (\frac{|u_{t}|^2}{2})_t\mbox{d}x=-\frac{\mbox{d}}{\mbox{d}t}\int\rho_t \frac{|u_{t}|^2}{2}\mbox{d}x+\int\rho_{tt}\frac{|u_{t}|^2}{2}\mbox{d}x\\
&=-\frac{\mbox{d}}{\mbox{d}t}\int\rho_t \frac{|u_{t}|^2}{2}\mbox{d}x-\int\mbox{div}(\rho u)_t\frac{|u_{t}|^2}{2}\mbox{d}x\\
&=-\frac{\mbox{d}}{\mbox{d}t}\int\rho_t \frac{|u_{t}|^2}{2}\mbox{d}x-\int(\rho u)_t\cdot\nabla u_t\cdot u_{t}\mbox{d}x\\
&\leq -\frac{\mbox{d}}{\mbox{d}t}\int\rho_t \frac{|u_{t}|^2}{2}\mbox{d}x+\|\sqrt{\rho}\|_{\infty}\|\sqrt{\rho}u_t\|_2\|u_t\|_4\|\nabla u_t\|_4\\
&+\|u\|_{\infty}\|\rho_t\|_4\|u_t\|_4\|\nabla u_t\|_2\\
&\leq -\frac{\mbox{d}}{\mbox{d}t}\int\rho_t \frac{|u_{t}|^2}{2}\mbox{d}x+C(\|u_t\|_4\|\nabla u_t\|_4+\|u_t\|_4\|\nabla u_t\|_2)\\
&\leq -\frac{\mbox{d}}{\mbox{d}t}\int\rho_t \frac{|u_{t}|^2}{2}\mbox{d}x+C(\|u_t\|_4\|\nabla^2 u_t\|_2+\|u_t\|_4\|\nabla u_t\|_2)\\
&\leq -\frac{\mbox{d}}{\mbox{d}t}\int\rho_t \frac{|u_{t}|^2}{2}\mbox{d}x+C\|u_t\|_4[\|\sqrt{\rho}u_{tt}\|_2+\| u_t\|_4+\|H_t\|_4\\
&+\|\nabla u_t\|_2+\|\nabla H_t\|_2+\|\nabla^3 u\|_2+\|\nabla u\|_{\infty}+1]\\
&\leq -\frac{\mbox{d}}{\mbox{d}t}\int\rho_t \frac{|u_{t}|^2}{2}\mbox{d}x+\frac{1}{16}\|\sqrt{\rho}u_{tt}\|_2^2+C[\|u_t\|_4^2+\|\nabla u_t\|_2^2\\&+\|\nabla H_t\|_2^2+\|\nabla^3 u\|_2^2+\|\nabla u\|_{\infty}^2+1],
\end{split}
\end{align}
\begin{align}
\begin{split}
&-\int\nabla (\frac12|H|^2_t)\cdot u_{tt}\mbox{d}x=-\int\nabla H\cdot H_t\cdot u_{tt}\mbox{d}x-\int H\cdot \nabla H_t\cdot u_{tt}\mbox{d}x\\
&\leq\|\sqrt{\rho}u_{tt}\|_2\|\frac{1}{\sqrt{\rho}}\|_{\infty}\|\nabla H\|_4\|H_t\|_4+\|H\|_{\infty}\|\nabla H_t\|_2\|\sqrt{\rho}u_{tt}\|_2\|\frac{1}{\sqrt{\rho}}\|_{\infty}\\
&\leq \frac{1}{16}\|\sqrt{\rho}u_{tt}\|_2^2+C\|\nabla H_t\|_2^2,
\end{split}
\end{align}
\begin{align}
\begin{split}
\int H_t\cdot\nabla H\cdot u_{tt}\mbox{d}x\leq \|\sqrt{\rho}u_{tt}\|_2\|\frac{1}{\sqrt{\rho}}\|_{\infty}\|\nabla H\|_4\|H_t\|_4
\leq \frac{1}{16}\|\sqrt{\rho}u_{tt}\|_2^2+C\|\nabla H_t\|_2^2,
\end{split}
\end{align}
\begin{align}
\begin{split}
\int H\cdot\nabla H_t\cdot u_{tt}\mbox{d}x\leq \|\sqrt{\rho}u_{tt}\|_2\|\frac{1}{\sqrt{\rho}}\|_{\infty}\| H\|_{\infty}\|\nabla H_t\|_2
\leq \frac{1}{16}\|\sqrt{\rho}u_{tt}\|_2^2+C\|\nabla H_t\|_2^2,
\end{split}
\end{align}
\begin{align}
\begin{split}
&-\int \rho_tu\cdot\nabla u\cdot u_{tt}\mbox{d}x\\&=-\frac{\mbox{d}}{\mbox{d}t}\int\rho_tu\cdot\nabla u\cdot u_{t}\mbox{d}x+\int\rho_{tt}u\cdot\nabla u\cdot u_{t}\mbox{d}x
+\int\rho_tu_t\cdot\nabla u\cdot u_{t}\mbox{d}x\\&+\int\rho_tu\cdot\nabla u_t\cdot u_{t}\mbox{d}x\\
&\leq -\frac{\mbox{d}}{\mbox{d}t}\int\rho_tu\cdot\nabla u\cdot u_{t}\mbox{d}x+\|\rho_{tt}\|_2\|u\|_{\infty}\|\nabla u\|_4\|u_t\|_4
+\|\rho_{t}\|_4\|\nabla u\|_4\|u_t\|_4^2\\&+\|\rho_{t}\|_4\|u\|_{\infty}\|\nabla u_t\|_2\|u_t\|_4\\
&\leq -\frac{\mbox{d}}{\mbox{d}t}\int\rho_tu\cdot\nabla u\cdot u_{t}\mbox{d}x+C(\|\rho_{tt}\|_2\|u_t\|_4+\|u_t\|_4^2+\|\nabla u_t\|_2\|u_t\|_4)\\
&\leq-\frac{\mbox{d}}{\mbox{d}t}\int\rho_tu\cdot\nabla u\cdot u_{t}\mbox{d}x+C(\|\rho_{tt}\|_2^2+\|u_t\|_4^2+\|\nabla u_t\|_2^2),
\end{split}
\end{align}
\begin{align}
-\int\rho u\cdot \nabla u_t\cdot u_{tt}\mbox{d}x\leq \|\sqrt{\rho}u_{tt}\|_2\|\sqrt{\rho}u\|_{\infty}\|\nabla u_t\|_2\leq \frac{1}{16}\|\sqrt{\rho}u_{tt}\|_2^2+C\|\nabla u_t\|_2^2,
\end{align}
\begin{align}
-\int\rho u_t\cdot \nabla u\cdot u_{tt}\mbox{d}x\leq \|\sqrt{\rho}u_{tt}\|_2\|\sqrt{\rho}\|_{\infty}\|\nabla u\|_4\|u_t\|_4\leq \frac{1}{16}\|\sqrt{\rho}u_{tt}\|_2^2+C\| u_t\|_4^2
\end{align}
and
\begin{align}
\int\lambda(\rho)_{tt}\mbox{div}u\mbox{div}u_{t}\mbox{d}x\leq \|\lambda(\rho)_{tt}\|_2\|\nabla u\|_{\infty}\|\nabla u_t\|_2\leq \frac{1}{2}(\|\lambda(\rho)_{tt}\|_2^2+\|\nabla u\|_{\infty}^2\|\nabla u_t\|_2^2).
\end{align}
Collecting all the above estimates and plugging them into \eqref{b4} yields that
\begin{align}\label{collect}
\begin{split}
&\frac12\|\sqrt{\rho}u_{tt}\|_2^2+\frac{\mbox{d}}{\mbox{d}t}\int\mu|\nabla u_t|^2+(\mu+\lambda(\rho))|\mbox{div}u_t|^2+\lambda(\rho)_t\mbox{div}u\mbox{div}u_{t}
\\&-p(\rho)_t\mbox{div}u_t+\rho_t\frac{|u_t|^2}{2}+\rho_tu\cdot\nabla u\cdot u_t\mbox{d}x\\
&\leq C[\|(\rho_{tt}, p(\rho)_{tt},\lambda(\rho)_{tt})\|_2^2+\| u_t\|_4^2+\|H_t\|_4^2+\|\nabla H_t\|_2^2+\|\nabla^3 u\|_2^2\\
&+(\|\nabla u\|_{\infty}^2+1)(\|\nabla u_t\|_2^2+1)].
\end{split}
\end{align}
Notice that
\begin{align*}
|\int\lambda(\rho)_t\mbox{div}u\mbox{div}u_{t}\mbox{d}x|&\leq \|\lambda(\rho)_{t}\|_4\|\mbox{div}u\|_4\|\nabla u_t\|_4\\
&\leq \frac{\mu}{8}\|\nabla u_t\|_2^2+C\|\lambda(\rho)_{t}\|_4^2\|\nabla u\|_4^2\\
&\leq \frac{\mu}{8}\|\nabla u_t\|_2^2+C,
\end{align*}
\begin{align*}
|-\int p(\rho)_t\mbox{div}u_t\mbox{d}x|\leq \frac{\mu}{8}\|\nabla u_t\|_2^2+C\|p(\rho)_{t}\|_2^2\leq \frac{\mu}{8}\|\nabla u_t\|_2^2+C,
\end{align*}
\begin{align*}
|\int\rho_t\frac{|u_t|^2}{2}\mbox{d}x|&=|\int\mbox{div}(\rho u)\frac{|u_t|^2}{2}\mbox{d}x|=|\int\rho u\cdot\nabla u_t\cdot u_t\mbox{d}x|\\
&\leq \|\sqrt{\rho }u_t\|_2\|\sqrt{\rho}u\|_{\infty}\|\nabla u_t\|_2\leq \frac{\mu}{8}\|\nabla u_t\|_2^2+C,
\end{align*}
\begin{align*}
&|\int\rho_tu\cdot\nabla u\cdot u_t\mbox{d}x|\\&=|\int\mbox{div}(\rho u)(u\cdot\nabla u\cdot u_t)\mbox{d}x|=|\int\rho u\cdot\nabla (u\cdot\nabla u\cdot u_t)\mbox{d}x|\\
&\leq \|\sqrt{\rho}u_t\|_2\|\sqrt{\rho}u\|_{\infty}(\|\nabla u\|_4^2+\|u\|_{\infty}\|\nabla^2u\|_2)+\|\rho|u|^2\|_{\infty}\|\nabla u_t\|_2\|\nabla u\|_2\\
&\leq \frac{\mu}{8}\|\nabla u_t\|_2^2+C.
\end{align*}
Thus, for some positive constant $C$, $C_1$, we have
\begin{align}\label{gt}
C_1(\|\nabla u_t\|_2^2-1)\leq G(t)\leq C(\|\nabla u_t\|_2^2+1),
\end{align}
with $G(t)=\int\mu|\nabla u_t|^2+(\mu+\lambda(\rho))|\mbox{div}u_t|^2+\lambda(\rho)_t\mbox{div}u\mbox{div}u_{t}
-p(\rho)_t\mbox{div}u_t+\rho_t\frac{|u_t|^2}{2}+\rho_tu\cdot\nabla u\cdot u_t\mbox{d}x$.
From \eqref{collect}, it holds that
\begin{align}\label{collect1}
\begin{split}
\frac12\|\sqrt{\rho}u_{tt}\|_2^2+\frac{\mbox{d}}{\mbox{d}t}G(t)
&\leq C[\|(\rho_{tt}, p(\rho)_{tt},\lambda(\rho)_{tt})\|_2^2+\| u_t\|_4^2+\|H_t\|_4^2\\
&+\|\nabla H_t\|_2^2+\|\nabla^3 u\|_2^2+(\|\nabla u\|_{\infty}^2+1)(G(t)+1)].
\end{split}
\end{align}
Multiplying the inequality \eqref{collect1} by $t$ and integrating the resulted equation in $t$ over $[\tau, t_1]$ with $\tau,\  t_1\in[0,T]$ give that
\begin{align}\label{collect2}
\begin{split}
&\int_{\tau}^{t_1}t\|\sqrt{\rho}u_{tt}\|_2^2(t)\mbox{d}t+t_1G(t_1)
\\&\leq C\tau G(\tau)+C\int_{\tau}^{t_1}[\|(\rho_{tt}, p(\rho)_{tt},\lambda(\rho)_{tt})\|_2^2+\| u_t\|_4^2+\|H_t\|_4^2+\|\nabla H_t\|_2^2\\
&+\|\nabla^3 u\|_2^2+G(t)]\mbox{d}t+C\int_{\tau}^{t_1}[(\|\nabla u\|_{\infty}^2+1)(tG(t)+1)]\mbox{d}t.
\end{split}
\end{align}
It follows from Lemma \ref{lem4.3} and \eqref{gt} that $G(t)\in L^1(0,T)$. Hence, due to \cite{kim}, there exists a subsequence $\tau_k$ such that
\begin{align*}
\tau_k\rightarrow0,\ \tau_kG(\tau_k)\rightarrow0,\ \mbox{as}\ k\rightarrow+\infty.
\end{align*}
Take $\tau=\tau_k$ in \eqref{collect2}, then $k\rightarrow+\infty$ and applying the Gronwall's inequality, one shows that
\begin{align}
\sup_{t\in[0,T]}[t\|\nabla u_t\|_2^2]+\int_0^Tt\|\sqrt{\rho}u_{tt}\|_2^2\mbox{d}t\leq C.
\end{align}
Note that \eqref{lem4.43} gives that
\begin{align}
\sup_{t\in[0,T]}[t\|(\rho_{tt}, p(\rho)_{tt},\lambda(\rho)_{tt})\|_2^2]+\int_0^Tt\|\nabla^2u_{t}\|_2^2\mbox{d}t\leq C.
\end{align}
Since $u_t=(L,B)^t-u\cdot\nabla u+\frac{1}{\rho}H\cdot\nabla H$, we prove that
$$\nabla(L,B)^t=\nabla u_t+\nabla(u\cdot\nabla u)-\nabla(\frac{1}{\rho}H\cdot\nabla H).$$
Consequently, it holds that
\begin{align}
\sup_{t\in[0,T]}[t\|\nabla(L,B)^t\|_2^2]\leq C,
\end{align}
which together with \eqref{bl3.74}  implies that
\begin{align}
\sup_{t\in[0,T]}[t\|(L,B)^t\|_2^2]\leq C\sup_{t\in[0,T]}[t\|\nabla(L,B)^t\|_2^2]\leq C.
\end{align}
Thus it holds that
\begin{align}
\sup_{t\in[0,T]}[t\|u_t\|_2^2]\leq C\sup_{t\in[0,T]}[t\|(L,B)^t\|_2^2+t\|u\cdot\nabla u\|_2^2+t\|\frac{1}{\rho}H\cdot\nabla H\|_2^2]\leq C.
\end{align}
Clearly,
\begin{align}
\sup_{t\in[0,T]}[t\|u_t\|_{H^1}^2]+\int_0^Tt\|u_t\|_{H^2}^2\mbox{d}t\leq C.
\end{align}
Applying $\partial_{x_ix_j}$, $i,\ j=1,\ 2$, to the continuity equation in \eqref{eq_EP_1} leads to
\begin{align*}
&(\rho_{x_ix_j})_t+u\cdot\nabla(\rho_{x_ix_j})+u_{x_ix_j}\cdot\nabla\rho
+u_{x_i}\cdot\nabla\rho_{x_j}+\rho_{x_ix_j}\mbox{div}u+\rho_{x_i}(\mbox{div}u)_{x_j}\\
&+\rho_{x_j}(\mbox{div}u)_{x_i}+\rho(\mbox{div}u)_{x_ix_j}=0.
\end{align*}
Multiplying the above equation by $q|\nabla^2\rho|^{q-2}\rho_{x_ix_j}$ with $q>2$ given in Theorem \eqref{thm_main} and summing over $i,\ j=1,\ 2$, and then integrating the resulting equation with respect to $x$ over $\mathbb{T}^2$ shows that
\begin{align*}
\frac{\mbox{d}}{\mbox{d}t}\|\nabla^2\rho\|_q^q&\leq (q-1)\|\|\nabla u\|_{\infty}\|\nabla^2\rho\|_q^q+Cq\|\nabla^2\rho\|_q^{q-1}[\|\nabla\rho\|_{2q}\|\nabla^2u\|_{2q}\\
&+\|\nabla u\|_{\infty}\|\nabla\rho\|_{2q}+\|\rho\|_{\infty}\|\nabla^3u\|_q],
\end{align*}
which leads to
\begin{align}\label{nabla}
\begin{split}
\frac{\mbox{d}}{\mbox{d}t}\|\nabla^2\rho\|_q&\leq C[\|\nabla u\|_{\infty}\|\nabla^2\rho\|_q+\|\nabla\rho\|_{2q}\|\nabla^2u\|_{2q}+\|\rho\|_{\infty}\|\nabla^3u\|_q]\\
&\leq C[\|\|\nabla u\|_{\infty}\|\nabla^2\rho\|_q+\|\nabla^2u\|_{W^{1,q}},
\end{split}
\end{align}
with $q>2$. In a similar fashion as \eqref{nabla}, one can deduce that
\begin{align}\label{nabla1}
\frac{\mbox{d}}{\mbox{d}t}\|\nabla^2p\|_q\leq C[\|\nabla u\|_{\infty}\|\nabla^2p\|_q+\|\nabla^2u\|_{W^{1,q}}].
\end{align}
Apply $\partial_{x_i}$ with $i=1,\ 2$ to the elliptic system $$\mathcal{L}_{\rho}u=\rho u_t+\rho u\cdot\nabla u-H\cdot\nabla H+\nabla p(\rho)+\nabla(\frac12|H|^2)$$ to get
\begin{align*}
&\mathcal{L}_{\rho}u_{x_i}=-\nabla(\lambda(\rho)_{x_i}\mbox{div}u)+\rho_{x_i} u_t+\rho u_{x_it}+\rho_{x_i} u\cdot\nabla u+\rho u_{x_i}\cdot\nabla u+\rho u\cdot \nabla u_{x_i}\\
&+\nabla p(\rho)_{x_i}
+\nabla(\frac12|H|_{x_i}^2)-H_{x_i}\cdot\nabla H-H\cdot\nabla H_{x_i}:=\Psi.
\end{align*}
Then the standard elliptic estimates give that
\begin{align*}
\|\nabla u\|_{W^{2,q}}&\leq C[\|\nabla u\|_q+\|\Psi\|_q]\\
&\leq C[1+(\|\nabla u\|_{\infty}+1)\|(\nabla^2\rho, \nabla^2p)\|_q+\|\nabla^2u\|_{2q}+\|u_t\|_{W^{1,q}}+\|\nabla^2H\|_q],
\end{align*}
which incorporating \eqref{nabla} and \eqref{nabla1} implies that
\begin{align}\label{nabla2}
\begin{split}
\frac{\mbox{d}}{\mbox{d}t}\|(\nabla^2\rho,\nabla^2p)\|_q
&\leq C[1+(\|\nabla u\|_{\infty}+1)\|(\nabla^2\rho, \nabla^2p)\|_q+\|u\|_{H^3}+\|u_t\|_{H^1}\\
&+\|\nabla u_t\|_q+\|\nabla^2H\|_q].
\end{split}
\end{align}
From Lemma \ref{lem2}, we have
\begin{align*}
\int_0^T\|\nabla u_t\|_q\mbox{d}t\leq C\int_0^T\|\nabla^2 u_t\|_2\mbox{d}t\leq C\sup_{t\in[0,T]}[\sqrt{t}\|\nabla u_t\|_2]\int_0^Tt^{-\frac12}\mbox{d}t\leq C.
\end{align*}
Therefore, from \eqref{nabla2} and the Gronwall's inequality, it holds that
\begin{align}\label{nabla3}
\begin{split}
\|(\nabla^2\rho,\nabla^2p)\|_q
&\leq \|(\nabla^2\rho_0,\nabla^2p(\rho_0))\|_q+C\int_0^t(1+\|u\|_{H^3}+\|u_t\|_{H^1}\\
&+\|\nabla u_t\|_q)\mbox{d}s \times\mbox{e}^{C\int_0^t(\|\nabla u\|_{\infty}+1)\mbox{d}s}\\
&\leq C,
\end{split}
\end{align}
which then leads to
\begin{align*}
\sup_{t\in[0,T]}\|(\rho, p(\rho))\|_{W^{2,q}}\leq C.
\end{align*}
Hence Lemma \ref{lem4.5} is now proved.
\end{proof}
\section{Proof of Theorem \ref{thm_main}}
In this section we complete the proof of Theorem \ref{thm_main}.

\begin{proof}[Proof of Theorem \ref{thm_main}]
From the uniform bounds in Lemmas \ref{lem3.1}-\ref{lem3.7} and Lemma \ref{lem4.1}-\ref{lem4.5}, we can show the solution sequence $(\rho^{n}, u^{n}, H^{n})$ converges to a limit $(\rho, u, H)$ satisfying the same bounds as $(\rho^{n}, u^{n}, H^{n})$ when $n\rightarrow\infty$ and the limit $(\rho,u,H)$ is the uinque solution to the original problem \eqref{eq_EP_1}-\eqref{isentropic}. We omit the details for brevity. Now, we will prove that $(\rho, u, H)$ satisfy the bounds in Theorem \ref{thm_main} and $(\rho, u, H)$ is in fact a classical solution to \eqref{eq_EP_1}.

Note that
$$(u,H)\in L^2(0,T;H^3(\mathbb{T}^2))\times L^2(0,T;H^3(\mathbb{T}^2)),$$
$$(u_t,H_t)\in L^2(0,T;H^1(\mathbb{T}^2))\times L^2(0,T;H^1(\mathbb{T}^2)),$$
then by Sobolev embedding theorem, one obtains
$$u\in C([0,T];H^2(\mathbb{T}^2))\hookrightarrow C([0,T]\times\mathbb{T}^2),$$
$$H\in C([0,T];H^2(\mathbb{T}^2))\hookrightarrow C([0,T]\times\mathbb{T}^2).$$
After a similar argument, from $(\rho, p(\rho))\in L^{\infty}([0,T];W^{2,q}(\mathbb{T}^2))$ and $(\rho_t, p(\rho)_t)\in L^{\infty}([0,T];H^{1}(\mathbb{T}^2))$, clearly it holds that
$$(\rho, p(\rho))\in C([0,T];W^{1,q}(\mathbb{T}^2))\cap C([0,T];W^{2,q}_{weak}(\mathbb{T}^2)).$$
Thus, from Lemma \ref{lem4.5}, clearly $(\rho,p(\rho))\in C([0,T]; W^{2,q}(\mathbb{T}^2))$.
 Theorem \ref{thm_main} is proved.
\end{proof}

 \vspace{0.4cm}

\textbf{Acknowledgements} The authors would like to thank Prof. Y. Wang for his value discussion.


\begin{thebibliography}{m}


\bibitem{Bian} D. F. Bian, B. Q. Yuan, Local
well-posedness in critical spaces for compressible MHD equations, (2010), 1--30,
submitted.

\bibitem{Bian2}  D. F. Bian, B. Q. Yuan, Well-posedness in super critical Besov spaces for compressible MHD
equations, \emph{Int. J. Dynamical Systems and Differential Equations}, 3(2011), 383--399.

\bibitem{B1} D. Bresch, B. Desjardins, Existence of global weak solutiona for a 2D viscous shallow water
equations and convergence to the quasi-geostrophic model,
\emph{Commun. Math. Sci.}, 238(2003), 211--223.

\bibitem{B2} D. Bresch, B. Desjardins and Chi-Kun Lin, On some compressible fluid models: Korteweg, lubrication, and shallow water
systems, \emph{Comm. Partial Differential Equations},
28(2003), 843--868.

\bibitem{B3} D. Bresch, B. Desjardins and G. M\'{e}tivier, Recent mathematical results and open problems about shallow water
systems, Analysis and simulation of fluid dynamics, 15--31,
\emph{Adv. Math. Fluid Mech.}, Birkh\"{a}user, Basel, 2007.

\bibitem{chemin} J. Y. Chemin, N. Lerner, Flot de champs de
vecteurs non lipschitziens et \'{e}quations de Navier-Stokes,
\emph{J. Differential Equations}, 121(1992), 314--328.

\bibitem{chemin2} J. Y. Chemin, Perfect incompressible fluids,
Oxford University Press, New York, 1998.

\bibitem{chemin3} J. Y. Chemin, Th\'{e}or\`{e}mes d'unicit\'{e} pour le syst\`{e}me de Navier-Stokes tridimensionnel,
\emph{J. d'Analyse Math.}, 77(1999), 27--50.

\bibitem{4} G. Q. Chen, D. H. Wang, Global solutions of nonlinear magnetohydrodynamics with large initial data, \emph{J. Differential Equations}, 182(2002), 344--376.

\bibitem{5} G. Q. Chen, D. H. Wang, Existence and continuous dependence of large
solutions for the magnetohydrodynamics equations, \emph{Z. Angew.
Math. Phys.}, 54(2003), 608--632.

\bibitem{c1} Q. L. Chen, C. X. Miao, Z. F. Zhang, Well-posedness in critical Spaces for
 Compressible Navier-Stokes equations with density dependent viscosities, \emph{Rev. Mat. Iberoamericana}, 26(2010), 915--946.

\bibitem{c2} Q. L. Chen, C. X. Miao, Z. F. Zhang, On the Well-posedness for
 the viscous shallow water equations, \emph{SIAM Jour. Math. Anal.}, 40(2008), 443--474.

\bibitem{cho} Y. Cho, B. J. Jin, Blow-up of viscous heat-conducting
compressible flows, \emph{J. Math. Anal. Appl.}, 320(2)(2006),
819--826.

\bibitem{kim} Y. Cho, H. J. Choe, H. Kim, Unique solvability of the initial boundary value problems for compressible viscous fluid, \emph{J. Math. Pure. Appl.}, 83(2004),
243--275.

\bibitem{d1} R. Danchin, Local Theory in critical Spaces for
Compressible Viscous and Heat-Conductive Gases, \emph{Comm. Partial
Differential Equations}, 26(2001), 1183--1233.

\bibitem{d2} R. Danchin, Global Existence in critical Spaces for
Flows of Compressible Viscous and Heat-Conductive Gases, \emph{Arch.
Rational Mech. Anal}, 160(2001), 1--39.

\bibitem{d3} R. Danchin, Global Existence in critical Spaces for
 Compressible Navier-Stokes equations, \emph{Invent. Math. Anal.}, 141(2000), 579--614.

\bibitem{d4} R. Danchin, On the uniqueness in critical Spaces for
 Compressible Navier-Stokes equations, \emph{Nonlinear Differrential Equations Appl.}, 12(2005), 111--128.

\bibitem{d5} R. Danchin, Well-posedness in Critical Spaces for
barotropic viscous fluids with truly not constant density,
\emph{Comm. Partial Differential Equations}, 32(2007), 1373--1397.

\bibitem{d6} R. Danchin, Density-dependent incompressible viscous fluids in critical Spaces
, \emph{Proc. Roy. Soc. Edinburgh Sect.A}, 133(2003), 1311-1334.

\bibitem{1} J. F. Gerebeau, C. L. Bris, T. Lelievre, Mathematical Methods for the Magnetohydrodynamics of Liquid Metals, Oxford University Press, Oxford, 2006.

\bibitem{H1} B. Haspot, Cauchy problem for
 the viscous shallow water equations with a term of capillarity, to appear in \emph{Math. Models methods Appl. Sci.}, 2010, doi: 10.1142/S0218202510004532.

\bibitem{6} D. Hoff, E. Tsyganov, Uniqueness and continuous dependence of
weak solutions in compressible magnetohydrodynamics, \emph{Z. Angew. Math.
Phys.}, 56(2005), 215--254.

 \bibitem{hoff} D. Hoff, K. Zumbrun, Multi-dimensional diffusion waves for the Navier-Stokes equations of compressible flow,
 \emph{Indiana Univ. Math. J.}, 44(1995), 603--676.

\bibitem{n.i} N. Itaya, On initial value problem of the motion of compressible viscous fluid,
especially on the problem of uniqueness, \emph{J. Math. Kyoto Univ.}, 16(1976), 413--427.
\bibitem{xin2} Q. S. Jiu, Y. Wang and Z. P. Xin, Global well-posedness of 2D compressible
Navier-Stokes equations with large data and vacuum, arXiv: 1202.1382v1 [math.AP] 7 Feb 2012.

\bibitem{Kato} T. Kato, quasi-linear equations of evolution, with applications to partial differential equations, Lecture Notes in Math.,
\textbf{vol. 448}, Springer-Verlag, Berlin, 1975.

\bibitem{7} S. Kawashima, M. Okada, Smooth global solutions for the
one-dimensional equations in magnetohydrodynamics, \emph{Proc. Japan Acad.
Ser. A, Math. Sci.}, 58(1982), 384--387.

\bibitem{k} T. Kobayashi and Y. Shibata, Dacay estimates of solutions for the equations of motion of compressible viscous
and heat-conductive gases in an exterior domain in $\mathbb{R}^{3}$,
\emph{Comm. Math. Phys.}, 200(1999), 621--660.

\bibitem{lions} P. L. Lion, Mathematics Topic in Fluid Mechanics, \textbf{vol.
2}, Oxford Lecture Series in Mathematics and its Applications,
Clarendon Press, Oxford, 1998.

\bibitem{2} Ta-tsien Li, Tiehu Qin, Physics and Partial Differential Equations,\textbf{ vol. I}, 2nd ed., \emph{Higher
Education Press}, Beijing, P. R. China, 2005.

\bibitem{3} R. Moreau, Magnetohydrodynamics, Kluwer Academic Publishers, Dordredht, 1990.

\bibitem{m} A. Matsumura and T. Nishida, The initial value problem for the equations of motion of viscous
and heat-conductive gases, \emph{J. Math. Kyoto Univ.}, 20(1980), 67--104.

\bibitem{a.m} A. Mellet, A. Vasseur, On the barotropic  compressible Navier-Stokes equations, \emph{Comm. Partial Differential Equations},
32(2007), 431--452.

\bibitem{J.N} J. Nash, Le probl\`{e}me de Cauchy pour les
\'{e}quations diff\'{e}rentielles d'un fluide g\'{e}n\'{e}ral,
\emph{Bull. Soc. Math. France},
90(1962), 487--497.

\bibitem{Novotny} A. Novotny, I. Stra\^{s}kraba, Introduction to the mathematical theorey  of compressible flow, Oxford Lecture Ser. Math. Appl. 27, Oxford Univ. Press, Oxford, 2004.

\bibitem{Pokhozhaev} S. I. Pokhozhaev, On S. L. Sobolev's embedding theorem in the case $pl=n$, \emph{Dokl. Nauchno-Tekhn. Konf. M\`{E}I}, (1965), 158--170.

\bibitem{LO} O. Rozanova, Blow up of smooth solutions to the
compressible Navier-Stokes equations with the data highly decreasing
at infinity, \emph{J. Differential Equations}, 245(2008), 1762--1774.

\bibitem{Solonnikov} V. A. Solonnikov, On solvability of an initial-boundary value problem for the equations of motion of viscous compressible fluid, \emph{Studies on linear operators and function theorey}, 6(1976), 128--142.
\bibitem{Tani} A. Tani, On the first initial-boundary value problem compressible viscous fluid motion, \emph{Publ. Res. Ins. Math. Sci.} 13(1977), 193--253.

\bibitem{Talenti} G. Talenti, Best constant in Sobolev inequality,\emph{ Ann. Mat. Pura Appl.} 110(4)(1976), 353--372.

\bibitem{Karzhikhov} V. A. Vaigant, A. V. Kazhikhov, On the existence of global solutions of two-dimensional Navier-Stokes equations of a compressible viscous fluid.
    (Russian)\emph{ Sibirsk. Mat. Zh.}, 36(6)(1995),1283--1316, ii; translation in \emph{Siberian Math. J.}, 36(6)(1995), 1108--1141.

\bibitem{w} W. K. Wang, C. J. Xu, The Cauchy problem for viscous shallow water
equations, \emph{Rev. Mat. Iberoamericana}, 21(2005), 1--24.


\bibitem{xin} Z. P. Xin, Blow up of smooth solutions to the compressible Navier-Stokes equation with compact density, \emph{Commun. Pure Appl. Anal.}, 51(1998), 229--240.

\bibitem{Yudovich} V. I. Yudovich, On some estimates connected with integral operators and solutions of elliptic equations, \emph{Dokl. Akad. Nauk SSSR}, 138(4)(1961), 805--808.


\end{thebibliography}
\end{document}